\documentclass[reqno,12pt]{amsart}

\usepackage[utf8]{inputenc}
\usepackage{microtype}
\usepackage{amsfonts}
\usepackage{amsmath}
\usepackage{graphicx}
\usepackage{hyperref}

\usepackage{amssymb}
\usepackage{mathrsfs}
\usepackage{mathtools}
\usepackage{enumitem}
\usepackage{comment}
\usepackage{esvect}

\usepackage{comment}
\usepackage{color}
\usepackage{tikz}
\usetikzlibrary{topaths}
\usetikzlibrary{tikzmark}
\usetikzlibrary{intersections}
\usetikzlibrary{decorations.pathreplacing}

\usepackage{a4wide}

\usepackage{empheq}
\numberwithin{equation}{section}

\makeatletter
\newcommand\@@vertwv[3]{\mathop{\ooalign{%
  $#1#3$\cr
  \hfil$#1\shortmid$\hfil\cr
  \hfil\raisebox{#2}{$#1\shortmid$}\hfil\cr}}}
\newcommand\@vertwv[1]{\mathchoice
  {\@@vertwv{\displaystyle     }{0.38ex}{#1}}%
  {\@@vertwv{\textstyle        }{0.38ex}{#1}}%
  {\@@vertwv{\scriptstyle      }{0.28ex}{#1}}%
  {\@@vertwv{\scriptscriptstyle}{0.21ex}{#1}}}
\newcommand\vertwedge{\@vertwv\wedge}
\newcommand\vertvee  {\@vertwv\vee  }
\makeatother

\newtheorem{theorem}{Theorem}[section]
\newtheorem{lemma}[theorem]{Lemma}
\newtheorem{proposition}[theorem]{Proposition}
\newtheorem{corollary}[theorem]{Corollary}

\newtheorem{question}[theorem]{Question}
\newtheorem{cit}[theorem]{Citation}
\newtheorem{observation}[theorem]{Observation}

\newtheorem*{main:aaut_clone}{Theorem~\ref{thrm:aaut_clone}}
\newtheorem*{main:nek_clone}{Corollary~\ref{cor:nek_clone}}
\newtheorem*{main:Finfty_inherited}{Corollary~\ref{cor:Finfty_inherited}}
\newtheorem*{main:good_H}{Theorem~\ref{thrm:good_H}}

\theoremstyle{definition}
\newtheorem{definition}[theorem]{Definition}
\newtheorem{remark}[theorem]{Remark}
\newtheorem{example}[theorem]{Example}

\newcommand{\Z}{\mathbb{Z}}
\newcommand{\N}{\mathbb{N}}

\newcommand{\R}{\mathbb{R}}
\newcommand{\lk}{\operatorname{lk}}

\newcommand{\dlk}{\operatorname{lk}^\downarrow}
\newcommand{\tree}{\mathcal{T}}

\DeclareMathOperator{\Stab}{Stab}

\DeclareMathOperator{\id}{id}

\DeclareMathOperator{\Aut}{Aut}
\DeclareMathOperator{\AAut}{AAut}
\DeclareMathOperator{\Grig}{Grig}
\DeclareMathOperator{\core}{core}

\newcommand{\F}{\mathrm{F}}

\newcommand{\clone}{\kappa}
\newcommand{\symmclone}{\varsigma}
\newcommand{\feet}{\phi}

\newcommand{\defeq}{\mathbin{\vcentcolon =}}
\newcommand{\eqdef}{\mathbin{=\vcentcolon}}

\usepackage{ifthen}
\usepackage{xargs}
\newcommand{\optionalarg}[2]{
\ifthenelse{\equal{#2}{}}{%
#1}{%
#1(#2)}
}

\newcommand{\Thomp}[1]%
   {\optionalarg{\mathscr{T}}{#1}}                 
   
\newcommand{\Thkern}[1]%
   {\optionalarg{\mathscr{K}}{#1}}                 

\newcommand{\Groupoid}[1]%
   {\optionalarg{\mathscr{G}}{#1}}                 
   
\newcommand{\Stein}[1]%
   {\optionalarg{\mathscr{X}}{#1}}                 
	
\newcommand{\dlkmodel}[2]                          
   {\optionalarg{\mathscr{L}_{#2}}{#1}}

\newcommand{\Fbr}%
   {F_{\operatorname{br}}}                 

\newcommand{\Vbr}%
   {V_{\operatorname{br}}}                 
   
\newcommand{\Vmock}%
   {V_{\operatorname{mock}}}                 
   
\newcommand{\Vloop}%
   {V_{\operatorname{loop}}}                 
   
\newcommand{\Floop}%
   {F_{\operatorname{loop}}}                 

\begin{document}

\title{Almost-automorphisms of trees, cloning systems and finiteness properties}
\date{\today}


\author[R.~Skipper]{Rachel Skipper}
\address{Mathematics Institute, University of G{\"o}ttingen, Bunsenstrasse 3-5, 37073 G{\"o}ttingen, Germany}
\email{skipper.rachel.k@gmail.com}

\author[M.~C.~B.~Zaremsky]{Matthew C.~B.~Zaremsky}
\address{Department of Mathematics and Statistics, University at Albany (SUNY), Albany, NY 12222, USA}
\email{mzaremsky@albany.edu}

\thanks{The first author was partially supported by a grant from the Simons Foundation (\#245855 to Marcin Mazur).}

\begin{abstract}
 We prove that the group of almost-automorphisms of the infinite rooted regular $d$-ary tree $\tree_d$ arises naturally as the Thompson-like group of a so called $d$-ary cloning system. A similar phenomenon occurs for any R\"over--Nekrashevych group $V_d(G)$, for $G\le \Aut(\tree_d)$ a self-similar group. We use this framework to expand on work of Belk and Matucci, who proved that the R\"over group, using the Grigorchuk group for $G$, is of type $\F_\infty$. Namely, we find some natural conditions on subgroups of $G$ to ensure that $V_d(G)$ is of type $\F_\infty$, and in particular we prove this for all $G$ in the infinite family of \v Suni\'c groups. We also prove that if $G$ is itself of type $\F_\infty$ then so is $V_d(G)$, and that every finitely generated virtually free group is self-similar, so in particular every finitely generated virtually free group $G$ yields a type $\F_\infty$ R\"over--Nekrashevych group $V_d(G)$.
\end{abstract}

\maketitle
\thispagestyle{empty}

\section*{Introduction}\label{sec:intro}

Let $\tree_d$ be the infinite rooted regular $d$-ary tree, and $\AAut(\tree_d)$ its group of almost-automorphisms. Given a self-similar group of automorphisms $G\le \Aut(\tree_d)$, the subgroup of $\AAut(\tree_d)$ that ``locally looks like'' $G$ is the R\"over--Nekrashevych group $V_d(G)$, sometimes called the Nekrashevych group. These groups were introduced in this degree of generality by Nekrashevych in \cite{nekrashevych04}. Before this, R\"over had considered $V_2(\Grig)$ for $\Grig$ the Grigorchuk group \cite{roever99}. The terms Nekrashevych group and R\"over--Nekrashevych group are also sometimes used to mean $V_d(G)$ only for the case of $G$ a finite subgroup of $S_d$, acting on $\tree_d$ by viewing its vertices as words in $\{1,\dots,d\}$. Here we will always use the term ``R\"over--Nekrashevych group,'' and will always allow any self-similar $G\le \Aut(\tree_d)$. The ``R\"over group'' is $V_2(\Grig)$.

Recall that a group is of \emph{type $F_n$} if it admits a classifying space with compact $n$-skeleton, and \emph{type $F_\infty$} if it is type $\F_n$ for all $n$. These \emph{finiteness properties} are a natural extension of the classical finiteness properties of being finitely generated (type $\F_1$) and finitely presentable (type $\F_2$). There is a large amount of literature devoted to finding finiteness properties of groups in the extended family of Thompson's groups. Usually the groups are of type $\F_\infty$, e.g., see \cite{belk16,brown87,bux16,farley15,fluch13,martinez-perez16,nucinkis18,thumann17}, though not always, e.g., Belk--Forrest's \emph{basilica Thompson group} $T_B$ \cite{belk15} is type $\F_1$ but not $\F_2$ \cite{witzel16}. Also, it is possible to build \emph{ad hoc} Thompson-like groups with arbitrary finiteness properties, using ``cloning systems'' \cite{witzel18}.

In this paper we prove that $\AAut(\tree_d)$ arises naturally from a ``$d$-ary cloning system'' on the wreath products $S_n \wr \Aut(\tree_d)$. A $d$-ary cloning system is a natural generalization of the notion of cloning system from \cite{witzel18}, which is the $d=2$ case. Similarly, for any self-similar $G\le\Aut(\tree_d)$ the R\"over--Nekrashevych group $V_d(G)$ arises from a $d$-ary cloning system on the groups $S_n \wr G$. The results are as follows (with the notation $\Thomp{-}$ explained in Subsection~\ref{sec:gen_cloning}).
\begin{main:aaut_clone}
 We have $\AAut(\tree_d)\cong\Thomp{S_*\wr\Aut(\tree_d)}$.
\end{main:aaut_clone}

\begin{main:nek_clone}
 For any self-similar $G\le\Aut(\tree_d)$, we have $V_d(G)\cong\Thomp{S_*\wr G}$.
\end{main:nek_clone}

Corollary~\ref{cor:nek_clone} allows us to prove the following:

\begin{main:Finfty_inherited}
 Let $G\le \Aut(\tree_d)$ be self-similar and of type $\F_\infty$. Then $V_d(G)$ is of type $\F_\infty$.
\end{main:Finfty_inherited}

See Remark~\ref{rmk:G_F_infty} for more thoughts on the proof of this result. In order to further populate the family of self-similar groups of type $\F_\infty$, we prove in Corollary~\ref{cor:vfree} that every finitely generated virtually free group is self-similar. In particular every finitely generated virtually free group yields a type $\F_\infty$ R\"over--Nekrashevych group.

We also find some sufficient conditions on $G$ to ensure that $V_d(G)$ is of type $\F_\infty$, even if $G$ is not. These conditions are the most general we could find for the strategy in \cite{belk16}, dealing with $V_2(\Grig)$, to still work. The result is as follows:

\begin{main:good_H}
 Let $G\le \Aut(\tree_d)$ be self-similar and suppose there exists $H\le G$ such that $H$ is nuclear in $G$ (see Definition~\ref{def:nuclear}) and $A$-coarsely self-similar (see Definition~\ref{def:Acss}), and $(H,A)$ is orderly (see Definition~\ref{def:orderly}) for some finite $A\subseteq V_d\cap\Aut(\tree_d)$. If $H$ is of type $\F_\infty$ then so is $V_d(G)$.
\end{main:good_H}

One of our main motivations here is that, since the $V_d(G)$ are often (virtually) simple (for a precise statement see \cite[Theorems~9.11 and~9.14]{nekrashevych04}), it is a pleasant surprise that $\Thomp{S_*\wr G}$ is isomorphic to $V_d(G)$, seeing as previous examples of Thompson groups arising from cloning systems tended to not be (virtually) simple. When the present paper was written it was still an open problem to find examples in the world of Thompson groups of simple groups of type $\F_n$ but not $\F_{n+1}$, for $n\ge 2$. After this paper was written, this problem was solved by the authors together with Stefan Witzel in \cite{skipper19}

This paper is organized as follows. First we establish our groups of interest in Section~\ref{sec:groups}. In Section~\ref{sec:cloning} we define $d$-ary cloning systems and prove that $\AAut(\tree_d)$ and the $V_d(G)$ all arise from such structures. In Sections~\ref{sec:large_cpx} and~\ref{sec:stein}, which are quite technical, we discuss complexes on which these groups act and find some conditions under which $V_d(G)$ is of type $\F_\infty$. Finally, we discuss our examples of particular interest in Section~\ref{sec:ex}.

\subsection*{Acknowledgments} We are grateful to Jim Belk and Francesco Matucci for supplying some helpful information about their paper \cite{belk16}, and to the referee for useful remarks and suggestions.


\section{The groups}\label{sec:groups}

In this section we introduce the groups we will study, all of which will be groups of self-similar automorphisms or almost-automorphisms of the infinite regular rooted $d$-ary tree $\tree_d$.

\subsection{Self-similar groups}\label{sec:selfsim}

Let $X$ be a finite set with $d$ elements, called an \emph{alphabet}, and let $X^*$ be the set of all finite words over $X$, including the empty word $\varnothing$. Then $X^*$ is the vertex set of the infinite rooted regular $d$-ary tree $\tree_d$, where two vertices $u$ and $v$ are connected by an edge whenever $ux=v$ or $vx=u$ for some $x\in X$.

An \emph{automorphism} of $\tree_d$ is a bijection from $X^*$ to $X^*$ that fixes the root and preserves incidence. Denote by $\Aut(\tree_d)$ the group of all automorphisms of $\tree_d$. As described in \cite[Subsection~1.3.2]{nekrashevych05}, an automorphism $f$ of $\tree_d$ can be described by labeling each element $v$ of $X^*$ by a permutation $f|_v\in S_d$. The automorphism $f$ is encoded by $\{f|_v \mid v \in X^*\}$ via the declaration that for a vertex $u=x_1x_2 \dots x_m \in X^*$, the action of $f$ on $u$ is given by
$$f(u)=f|_\varnothing(x_1) f|_{x_1}(x_2) f|_{x_1 x_2}(x_3)\cdots f|_{x_1\cdots x_{m-1}}(x_m)\text{.}$$

\begin{definition}[Wreath recursion, state]
 For $G$ a group we use the wreath product notation $S_d \wr G \defeq S_d \ltimes (G\times \cdots \times G)$, so the acting group is written on the left. The above discussion (see also \cite[Proposition~1.4.2]{nekrashevych05}) gives us an isomorphism $\Aut(\tree_d) \cong S_d \wr (S_d \wr (S_d \cdots))$ and in particular $\Aut(\tree_d)\cong S_d \wr \Aut(\tree_d)$. Under this isomorphism, an element $f\in \Aut(\tree_d)$ can be decomposed as
$$f=\sigma(f_1, \dots, f_d)\text{,}$$
 where $\sigma \in S_d$ and $f_i$ is the restriction of the permutation labeling of $f$ to the $i$th subtree rooted at the first level (canonically identified with the original $\tree_d$).  We refer to $f_i$ as the \emph{state of $f$ at the $i$th vertex}. Iterating this decomposition, for each $u\in X^*$ we obtain $f_u$, the \emph{state of $f$ at the vertex $u$}.
\end{definition}

Let $\rho \colon \Aut(\tree_d)\to S_d$ be the map $f\mapsto f|_\varnothing$, so the wreath recursion of $f$ is
$$f=\rho(f)(f_1,\dots,f_d)\text{.}$$

\begin{definition}[Self-similar]
 We say a group $G\le \Aut (\tree_d)$ is \emph{self-similar} if for all $g=\rho(g)(g_1, \dots, g_d)\in G$ and for all $i$, we have $g_i\in G$.
\end{definition}

A property that frequently arises in the subject of self-similar groups and will be relevant to the examples in Section~\ref{sec:ex} is that of contracting.

\begin{definition}[Contracting, nucleus]
 A self-similar group $G$ is called \emph{contracting} if there exists a finite subset $S\subseteq G$ such that for all $g\in G$ there exists $n\in\N$ with the property that for every vertex $u$ of length greater than or equal to $n$, the state of $g$ at $u$ is in $S$. The smallest such $S$ is called the \emph{nucleus} of $G$.
\end{definition}

\subsection{Up from finite index}\label{sec:finite_index}

In this subsection we prove that in many cases if an abstract group contains a finite index copy of a self-similar group of automorphisms of some tree then the group itself is isomorphic to a self-similar group of automorphisms of some (possibly other) tree. In other words, virtually self-similar often implies self-similar, and so examples of self-similar groups are more ubiquitous than one might expect.

For a group $G$, a \emph{virtual endomorphism} of $G$ is a homomorphism $\phi$ from a finite index subgroup $K$ of $G$ to $G$. The \emph{$\phi$-core} of $K$, denoted $\phi$-core$(K)$, is the maximal normal subgroup $\tilde{K}$ of $K$ that is $\phi$-invariant (i.e., $\phi(\tilde{K}) \leq \tilde{K}$). These $\phi$-cores always exist, as the next result shows.

\begin{cit}\cite[Proposition~2.7.5]{nekrashevych05}\label{cit:nek}
 $$\phi\text{-core}(K)=\bigcap_{m \geq 1} \bigcap_{g\in G} g^{-1} \text{Domain}(\phi^m)g\text{.}$$
\end{cit}

In words, the $\phi$-core is the normal core of the iterated pullback $\cdots \phi^{-1}\phi^{-1} K$. The virtual endomorphism $\phi$ is called \textit{simple} if the $\phi$-core is trivial.

Self-similar groups have a strong connection to virtual endomorphisms. In particular, if a group $G$ admits a simple virtual endomorphism then it is self-similar, as the next result shows.

\begin{cit}\cite[Theorem~6]{berlatto07}\label{cit:berlatto}
 Let $G$ be a group, $K$ a finite index subgroup of $G$, and $X$ a left transversal of $K$ in $G$. Let $\rho$ be the permutation representation of $G$ on $X$. In particular for $g\in G$ and $x\in X$, we have $\rho(g)(x)$ equals the $x'\in X$ satisfying $gxK=x'K$. Also let $\tree_d$ be the tree with vertex set $X^*$ (so $d=|X|$) and let $\phi$ be a virtual endomorphism from $K$ to $G$. Then the quadruple $(G, K, \phi, X)$ provides a representation $\psi\colon G\to \Aut(\tree_d)$ recursively defined by
$$\psi(g)=\rho(g)\{\psi(\phi(\rho(g)(x)^{-1}gx))\mid x\in X \}\text{.}$$
 Furthermore, $\ker(\psi)=\phi$-core($K$). In particular if $\phi$ is simple then $\psi$ is faithful.
\end{cit}

In \cite{berlatto07} this is phrased using different left/right conventions, but the above form is the one we want.

Conversely, assume that $G$ is a self-similar group. For a fixed $x\in X$, let $\phi_x$ be the projection $\Stab_G(x)\to G$ onto the $x$th coordinate. This produces a 3-tuple $(G,\Stab_G(x),\phi_x)$ where $\phi_x$ is a virtual endomorphism. In this case, the $\phi_x$-core is 
$$\bigcap_{m\ge 1}\bigcap_{g\in G} g^{-1}\Stab_G(x^m)g$$
where $x^m$ is the vertex $xx\cdots x$.

\begin{remark}
 It is likely that the virtual endomorphism $\phi_x$ will not recreate the known self-similar action of $G$. Instead, it may produce a different self-similar action on a possibly different tree. In fact, there are self-similar group actions which cannot be obtained from Citation~\ref{cit:berlatto}, such as the ones described in Example~\ref{ex:fss_gp}.
\end{remark}

Next, we show that self-similar actions can often be induced up from subgroups of finite index acting self-similarly.

\begin{proposition}[Inducing self-similar actions]\label{prop:induce}
 Let $G$ be a group and let $K$ be a finite index subgroup of $G$. Suppose that $(K,\tilde{K},\phi,X)$ defines a faithful self-similar action of $K$ on the tree with vertex set $X^*$ as in Citation~\ref{cit:berlatto}. Let $\tilde{\phi}$ be the homomorphism from $\tilde{K}$ to $G$ defined by
$$\tilde{K} \xrightarrow{\phi} K \hookrightarrow G$$
and let $Y$ be a right transversal to $\tilde{K}$ in $G$. Then the 4-tuple $(G,\tilde{K},\tilde{\phi},Y)$ defines a faithful self-similar action of $G$ on the tree with vertex set $Y^*$.
\end{proposition}

\begin{proof}
 By Citation~\ref{cit:berlatto}, the 4-tuple $(G, \tilde{K}, \tilde{\phi}, Y)$ defines a self-similar action of $G$ on the tree with vertex set $Y^*$. Thus it suffices to show the action is faithful. By Citation~\ref{cit:nek} and Citation~\ref{cit:berlatto}, the kernel of the action is 
$$\bigcap_{m \geq 1} \bigcap_{g\in G} g^{-1} \text{Domain}(\tilde{\phi}^m)g\leq \bigcap_{m \geq 1} \bigcap_{g\in K} g^{-1} \text{Domain}(\phi^m)g=\{1\}\text{.}$$
\end{proof}

The hypothesis of Proposition~\ref{prop:induce} occurs when the self-similar action of the finite index subgroup comes from a simple virtual endomorphism or when a self-similar group $K$ has the property that $\bigcap_{m \geq 1} \bigcap_{g\in K} g^{-1} \Stab_K(x^m) g =\{1\}$, for instance when $K$ acts transitively on the levels of the tree.

As an application of Proposition~\ref{prop:induce}, we use a result from \cite{steinberg11} to prove that every finitely generated virtually free group is self-similar.

Consider the following recursively defined automorphisms of $\tree_2$:
$$a=(0~1)(c,b), \quad b=(0~1)(b,c), \quad c=\sigma_0(d_1,d_1)$$
$$d_i=\sigma_i(d_{i+1}, d_{i+1}), 1\le i \le n-1 \text{ and } d_n=\sigma_n(a,a)$$
where an odd number of the $\sigma_i$'s are not the identity permutation.

\begin{cit}\cite{steinberg11}
 For $n$ even, the automorphisms $a$, $b$, $c$, $d_1,\dots,d_n$ generate a free group of rank $n+3$.
\end{cit}

\begin{lemma}\label{lem:level_transitive}
 Let $a$, $b$, $c$, $d_1,\dots,d_n$ be defined as above and suppose $\sigma_1=(0~1)$ and $\sigma_i=\id$ for $i\ne 1$. Then $a$ acts transitively on the levels of $\tree_2$.
\end{lemma}

\begin{proof}
 By \cite[Lemma~1.10.3]{nekrashevych05}, an automorphism $g$ of $\tree_2$ is transitive on all levels of $\tree_2$ if and only if for every $n\ge 0$ the number of vertices $v$ at level $n$ where $g|_v$ is non-trivial is odd. We claim this is the case for $g=a$.
 
 We induct on the level in $\tree_2$. Clearly, the number of vertices on level $0$ with non-trivial permutation label is odd since $a|_\varnothing=(0~1)$.
 
 Now suppose for $n\ge 1$, $a|_v$ is non-trivial for an odd number of vertices $v$ on the $n$th level. We claim that the parity of the number of non-trivial permutations does not change from level $n$ to level $n+1$. Indeed, fix $v$ a vertex on the $n$th level and let $v_0$ and $v_1$ be the two vertices adjacent to $v$ on level $(n+1)$. If $a_v\in \{a, b, c, d_2, \dots, d_n\}$ then $a|_v$ is non-trivial if and only if exactly one of $a|_{v_i}$ is non-trivial for $i=0,1$. In particular the parity of the number of non-trivial permutations is preserved when passing from $v$ to $v_0$ and $v_1$. The only remaining possibility is that $a_v=d_1$. In this case, $a|_v$ is non-trivial and $a|_{v_0}$ and $a|_{v_1}$ are either both trivial (if $n>4$) or both non-trivial (if $n=4$). But the number of $d_1$'s appearing as states of $a$ on a given level is visibly even, and so in this case too, the parity of the number of non-trivial permutations is preserved when passing from $v$ to $v_0$ and $v_1$. This completes the induction.
\end{proof}

\begin{corollary}\label{cor:vfree}
 Every finitely generated virtually free group $G$ admits a faithful self-similar action on some $\tree_d$.
\end{corollary}

\begin{proof}
 Every such $G$ is virtually $F_n$ for some odd $n>4$. By Lemma~\ref{lem:level_transitive} the hypothesis of Proposition~\ref{prop:induce} is met, using $K=F_n$. Proposition~\ref{prop:induce} now gives us the result.
\end{proof}

\subsection{Almost-automorphism groups}\label{sec:aaut}

In addition to $\Aut(\tree_d)$ we will also be interested in $\AAut(\tree_d)$, the group of \emph{almost-automorphisms} of $\tree_d$. An almost-automorphism of $\tree_d$ is not an automorphism of $\tree_d$, but rather a certain self-homeomorphism of the boundary $\partial \tree_d$ ``locally determined'' by automorphisms of $\tree_d$. To make this precise, we need some setup.

\begin{definition}[Complete subtree, $d$-caret, complement]\label{def:subtrees}
 A subtree $T$ of $\tree_d$ is called a \emph{rooted complete subtree} if it has the same root as $\tree_d$ and for each non-leaf vertex $v$ of $T$, every vertex of $\tree_d$ adjacent to $v$ lies in $T$. The \emph{$d$-caret} is the rooted complete subtree of $\tree_d$ where every non-root vertex is a leaf. The \emph{complement} $T^c$ of $T$ is the set theoretic complement $\tree_d \setminus \mathring{T}$, where $\mathring{T}$ is $T$ with its leaves removed (note that the interiors of the edges incident to the leaves remain in $\mathring{T}$).
\end{definition}

Let $T$ be a finite rooted complete subtree of $\tree_d$ with $n$ leaves. The complement $T^c$ is a forest consisting of $n$ disjoint copies of $\tree_d$. In particular if $T_-$ and $T_+$ are two finite rooted $d$-ary trees, each with $n$ leaves, then given $n$ automorphisms $f_1,\dots,f_n \in \Aut(\tree_d)$ and a permutation $\sigma\in S_n$, we can consider the bijection $\sigma(f_1,\dots,f_n) \colon T_+^c \to T_-^c$ given by, for each $1\le i\le n$, applying $f_i$ to the $i$th tree of $T_+^c$ and then sending it to the $\sigma(i)$th tree of $T_-^c$.

\begin{definition}[Almost-automorphism]
 A triple $(T_-,\sigma(f_1,\dots,f_n),T_+)$ with $T_\pm$ and $\sigma(f_1,\dots,f_n)$ as above induces a self-homeomorphism of $\partial \tree_d$, which we denote by
$$((T_-,\sigma(f_1,\dots,f_n),T_+))$$
and call an \emph{almost-automorphism} of $\tree_d$. We denote by $\AAut(\tree_d)$ the group of all almost-automorphisms of $\tree_d$.
\end{definition}

Two different triples may represent the same almost-automorphism, for instance
$$((T,(\id,\dots,\id),T))=((U,(\id,\dots,\id),U))$$
for any $T$ and $U$. More generally, \cite[Lemma~2.3]{leboudec17} says the following:

\begin{cit}\cite[Lemma~2.3]{leboudec17}\label{cit:same_aaut}
 The almost-automorphisms $((T_-,\sigma(f_1,\dots,f_n),T_+))$ and $((U_-,\tau(g_1,\dots,g_m),U_+))$ are equal if and only if there exist finite rooted complete subtrees $V_-$ and $V_+$ of $\tree_d$, with $T_-,U_-\subseteq V_-$ and $T_+,U_+\subseteq V_+$, such that $\sigma(f_1,\dots,f_n)$ and $\tau(g_1,\dots,g_m)$ restrict to bijections $V_+^c \to V_-^c$ and these bijections coincide.
\end{cit}

We recall here some important subgroups of $\AAut(\tree_d)$, namely the Higman--Thompson group and the family of R\"over--Nekrashevych groups. See also Definition~3.2 and Section~7.2 of \cite{leboudec17}. The notation $\vv{\id}$ denotes a tuple whose entries are all $\id$.

\begin{definition}[Higman--Thompson group]
 The \emph{Higman--Thompson group} $V_d$ is the subgroup of $\AAut(\tree_d)$ consisting of those almost-automorphisms of the form
$$((T_-,\sigma\vv{\id},T_+))\text{.}$$
\end{definition}

\begin{definition}[R\"over--Nekrashevych group]
 Let $G$ be a self-similar subgroup of $\Aut(\tree_d)$. The \emph{R\"over--Nekrashevych group} $V_d(G)$ is the subgroup of $\AAut(\tree_d)$ consisting of all almost-automorphisms of the form $((T_-,\sigma(g_1,\dots,g_n),T_+))$ for $g_i\in G$ and $\sigma\in S_n$.
\end{definition}

In \cite[Subsection~7.2]{leboudec17} $V_d(G)$ is denoted $\AAut_G(\tree_d)$. Sometimes $V_d(G)$ is called the \emph{Nekrashevych group} of $G$, and sometimes the terms Nekrashevych group and/or R\"over--Nekrashevych group are reserved for the case when $G$ is finite. Here we will always allow any self-similar $G$, and will use the term R\"over--Nekrashevych group. The \emph{R\"over group} is the specific group $V_2(\Grig)$ for $\Grig$ the Grigorchuk group, first introduced in \cite{roever99} and proved to be of type $\F_\infty$ in \cite{belk16}.

\begin{example}\label{ex:fss_gp}
 For any $D\le S_d$ we can embed $D$ into $\Aut(\tree_d)$ by viewing the vertices of $\tree_d$ as words in the alphabet $\{1,\dots,d\}$. Let $\iota\colon D\to \Aut(\tree_d)$ be this embedding, and let $G\defeq \iota(D)$. Then the wreath recursion of any $g\in G$ is $g=(\iota^{-1}(g))(g,\dots,g)$, so $G$ is self-similar, and it therefore makes sense to consider the R\"over--Nekrashevych group $V_d(G)$. This is an example of an \emph{FSS group}, as in \cite{hughes09,farley15} (see in particular \cite[Remark~2.13]{farley15}), and is of type $\F_\infty$ \cite[Corollary~6.6]{farley15}. 
\end{example}

\begin{example}\label{ex:sunic}
 Let $A$ be the cylic group of order $d$, say $A=\langle a\rangle$, and view $A$ as a subgroup of $\Aut(\tree_d)$ by letting $a$ permute just the first level of $\tree_d$. Let $B$ be a finitely generated group and $\rho\colon B\to B$ an automorphism, and suppose there exists a homomorphism $\omega\colon B\to A$ such that no non-trivial $\rho$-orbit lies entirely in $\ker(\omega)$. Map $B$ to $\Aut(\tree_d)$ by declaring the wreath recursion $b=(\omega(b),\id,\dots,\id,\rho(b))$. This defines a faithful action on $\tree_d$ since we have assumed that for any $1_B\ne b\in B$ there exists $n$ such that $\omega(\rho^n(b))\ne 1_A$. Viewing $B$ as a subgroup of $\Aut(\tree_d)$ in this way, define $G_{\omega,\rho}\defeq\langle A,B\rangle$. We call this the \emph{\v Suni\'c group} for $\omega$ and $\rho$. It is self-similar by construction. See Subsection~\ref{sec:sunic} for more details and background about these groups, including a proof that the R\"over--Nekrashevych group $V_d(G)$ is of type $\F_\infty$ for any \v Suni\'c group $G$. Note that the Grigorchuk group $\Grig$ is the \v Suni\'c group with $d=2$, $B=\{1_B,b,c,d\}\cong K_4$, $\omega\colon K_4\to A$ given by $b,c\mapsto a$, and $\rho\colon K_4\to K_4$ given by $b\mapsto c \mapsto d$.
\end{example}

\begin{remark}
 In \cite{sauer17} Sauer and Thumann prove that groups like $\AAut(\tree_d)$ and some of its relatives have the natural analog of the type $\F_\infty$ property for totally disconnected locally compact groups. It seems likely that one could use our methods to recover this result for $\AAut(\tree_d)$, and possibly extend it to certain of its non-discrete subgroups. For now we will focus just on discrete examples and the classical $\F_\infty$ property, and leave questions about totally disconnected groups for the future.
\end{remark}


\section{$d$-ary cloning systems}\label{sec:cloning}

In \cite{witzel18}, the framework of \emph{cloning systems} was established to encode certain families of groups into a Thompson-like group. Examples include the classical Thompson groups $F$ and $V$, the braided Thompson groups of Brin and Dehornoy \cite{brin07,dehornoy06}, Thompson groups for upper triangular matrix groups \cite{witzel18} and the generalized Thompson groups of Tanushevski \cite{tanushevski16}. In this section we establish $d$-ary analogs, which we call \emph{$d$-ary cloning systems}, and show that $\AAut(\tree_d)$ and any R\"over--Nekrashevych group $V_d(G)$ arise as Thompson-like groups of $d$-ary cloning systems. We remark that these particular cloning systems seem somewhat related to examples developed by Tanushevski \cite{tanushevski16}, and it would be interesting to pin down the connection.

\subsection{General $d$-ary cloning systems}\label{sec:gen_cloning}

The setup in \cite{witzel18} only accounted for Thompson structures coming from binary trees. Since we are dealing with $d$-ary trees, we need to establish cloning systems for the $d$-ary situation. To avoid a long digression into the abstract world of Brin--Zappa--Sz\'ep products, we will not follow the construction in \cite{witzel18} but rather the shorter introduction to cloning systems in \cite{zaremsky18}. This has the additional advantage that certain data from \cite{witzel18}, namely a family of maps $(\iota_{m,n})_{m\le n}$, will be unnecessary. As a remark, similar, and many additional, improvements were also made by Witzel in \cite{witzel17}, also eliminating the $(\iota_{m,n})_{m\le n}$, and developing a more categorical framework for cloning systems.

Let $d\ge 2$ be an integer and $(G_n)_{n\in\N}$ a family of groups. Suppose we have a \emph{representation map} $\rho_n \colon G_n \to S_n$ to the symmetric group $S_n$, for each $n\in\N$, and an injective function $\clone_k^n \colon G_n \to G_{n+d-1}$ (not necessarily a homomorphism) for each pair $(k,n)$ with $1\le k\le n$. To be consistent with \cite{witzel18}, we will write the $\rho_n$ to the left and $\clone_k^n$ to the right of their arguments, so expressions like $\rho_n(g)$ and $(g)\clone_k^n$ will be standard.

\begin{definition}
 The triple $((G_n)_{n\in\N},(\rho_n)_{n\in\N},(\clone_k^n)_{k\le n})$ is called a \emph{$d$-ary cloning system}, and the functions $\clone_k^n$ are \emph{$d$-ary cloning maps}, provided the following axioms hold:
 
 \textbf{(C1):} (\emph{Cloning a product}) $(gh)\clone_k^n = (g)\clone_{\rho_n(h)k}^n (h)\clone_k^n$.
 
 \textbf{(C2):} (\emph{Product of clonings}) $\clone_\ell^n \circ \clone_k^{n+d-1} = \clone_k^n \circ \clone_{\ell+d-1}^{n+d-1}$.
 
 \textbf{(C3):} (\emph{Compatibility}) $\rho_{n+d-1}((g)\clone_k^n)(i) = (\rho_n(g))\symmclone_k^n(i)$ for all $i\ne k,k+1,\dots,k+d-1$.
 
 Here we always have $1\le k<\ell \le n$ and $g,h\in G_n$, and $\symmclone_k^n$ denotes the standard $d$-ary cloning maps for the symmetric groups, which we discuss in Example~\ref{ex:symm_gps}.
\end{definition}

For the sake of giving the reader some intuition, we summarize some of the background from \cite{witzel18} that has been swept under the rug. The $\rho_n$ reflect a left action of the $G_n$ on the so called forest monoid, or in this context the monoid of $d$-ary forests, and the $\clone_k^n$ reflect a right action of this monoid on the $G_n$ (or rather on their direct limit under a directed system of inclusions $\iota_{m,n}$, which for our current level of detail we are able to ignore). These mutual actions describe how to ``pull'' a $d$-caret ``through'' an element of $G_n$, namely they dictate where the $d$-caret emerges and which element of $G_{n+d-1}$ is produced. While the $\clone_k^n$ are not group homomorphisms, (C1) shows that they are sort of ``twisted'' homomorphisms, twisted via the action from the $\rho_n$. The cloning axiom (C2) mirrors the defining relations in the $d$-ary forest monoid, and (C3) ensures that the $\rho_n$ and $\clone_k^n$ behave well enough together that a Thompson-like group can be naturally constructed.

\begin{example}(Symmetric groups)\label{ex:symm_gps}
 Consider the family $(S_n)_{n\in\N}$ of symmetric groups. Let $\rho_n \colon S_n \to S_n$ be the identity and define $\symmclone_k^n \colon S_n \to S_{n+d-1}$ via
$$((\sigma)\symmclone_k^n)(i) \defeq \left\{\begin{array}{ll} \sigma(i) & \text{if } i\le k \text{ and } \sigma(i)\le \sigma(k)\\ \sigma(i)+d-1 & \text{if } i<k \text{ and } \sigma(i)>\sigma(k)\\ \sigma(i-d+1) & \text{if } i>k+d-1 \text{ and } \sigma(i-d+1)<\sigma(k)\\ \sigma(i-d+1)+d-1 & \text{ if } i\ge k+d-1 \text{ and } \sigma(i-d+1)\ge \sigma(k)\\ \sigma(k)+i-k & \text{ if } k<i<k+d-1 \text{.}\end{array}\right.$$
Then $((S_n)_{n\in\N},(\id_{S_n})_n,(\symmclone_k^n)_{k\le n})$ is a $d$-ary cloning system. There is nothing to prove for (C3), and proving (C1) and (C2) is straightforward but tedious, so we will leave this as an exercise. We refer the reader to \cite[Example~2.9]{witzel18} for a complete proof in the the $d=2$ case.
\end{example}

A picture will make this much more clear: see Figure~\ref{fig:symm_clone} for an example of $\symmclone_k^n$.

\begin{figure}[htb]
 \centering
 \begin{tikzpicture}[line width=0.8pt]
  \draw[->] (0,-2) -- (1,0); \draw[->] (1,-2) -- (2,0); \draw[->,dashed] (2,-2) -- (0,0);
  \node at (3,-1) {$\stackrel{\symmclone_3^3}{\longrightarrow}$};
  \node at (0,-2.25) {$1$};   \node at (0,0.25) {$1$};   \node at (1,-2.25) {$2$};   \node at (1,0.25) {$2$};   \node at (2,-2.25) {$3$};   \node at (2,0.25) {$3$};
  \begin{scope}[xshift=4cm]
   \draw[->] (0,-2) -- (3,0); \draw (1,-2)[->] -- (4,0); \draw (2,-2)[->,dashed] -- (0,0); \draw (3,-2)[->,dashed] -- (1,0); \draw (4,-2)[->,dashed] -- (2,0);
   \node at (0,-2.25) {$1$};   \node at (0,0.25) {$1$};   \node at (1,-2.25) {$2$};   \node at (1,0.25) {$2$};   \node at (2,-2.25) {$3$};   \node at (2,0.25) {$3$};   \node at (3,-2.25) {$4$};   \node at (3,0.25) {$4$};   \node at (4,-2.25) {$5$};   \node at (4,0.25) {$5$};
  \end{scope}
 \end{tikzpicture}
 \caption{An example of $3$-ary cloning in the symmetric groups. Here we see that $(1~2~3)\symmclone_3^3 = (1~4~2~5~3)$. The arrow getting ``cloned'' is dashed, as are the three arrows resulting from the cloning.}
 \label{fig:symm_clone}
\end{figure}
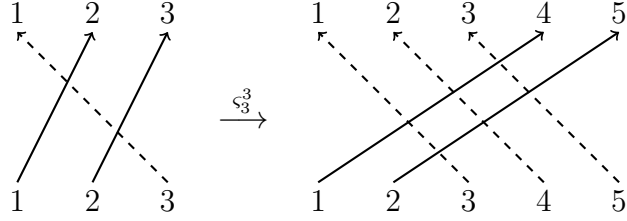

Note that we are not including ``$d$'' in the notation $\symmclone_k^n$, since it will always be clear from context, but the reader should be aware that by $\symmclone_k^n$ here we always mean the $d$-ary cloning maps on the symmetric groups, for whichever $d$ is relevant, and not just the $2$-ary cloning maps denoted $\symmclone_k^n$ in \cite{witzel18}.
 
Given a $d$-ary cloning system on a family of groups $(G_n)_{n\in\N}$, we can construct a Thompson-like group $\Thomp{G_*}$, as in \cite[Subsection~2.4]{witzel18}. An element of $\Thomp{G_*}$ is represented by a triple $(T_-,g,T_+)$, where $T_\pm$ are finite rooted $d$-ary trees with the same number of leaves, say $n$, and $g\in G_n$. (Note that the number of leaves must be of the form $1+r(d-1)$ for some $r\in\N\cup\{0\}$, so the $G_n$ for $n$ not of this form do not actually go into the construction of $\Thomp{G_*}$.) Two such triples $(T_-,g,T_+)$ and $(U_-,h,U_+)$ represent the same element of $\Thomp{G_*}$ if we can obtain one from the other via a sequence of expansions and reductions. An \emph{expansion} of $(T_-,g,T_+)$ is a triple $(T_-',(g)\clone_k^n,T_+')$, where $T_+'$ is $T_+$ with a $d$-caret added to its $k$th leaf and $T_-'$ is $T_-$ with a $d$-caret added to its $\rho_n(g)(k)$th leaf. A \emph{reduction} is the reverse of an expansion. The equivalence class of $(T_-,g,T_+)$ under expansions and reductions is denoted $[T_-,g,T_+]$.

For example the $d$-ary cloning system $((S_n)_{n\in\N},(\id_{S_n})_n,(\symmclone_k^n)_{k\le n})$ gives rise to the Higman--Thompson group $\Thomp{S_*}\cong V_d$.

There is a natural notion of a \emph{subsystem} of a $d$-ary cloning system, which we define next.

\begin{definition}[Cloning subsystem]
 Let $((G_n)_{n\in\N},(\rho_n)_{n\in\N},(\clone_k^n)_{k\le n})$ be a $d$-ary cloning system. Let $(H_n)_{n\in\N}$ be a family of subgroups $H_n\le G_n$ such that $(H_n)\clone_k^n \subseteq H_{n+d-1}$ for all $k\le n$. The families of restrictions $\rho_n'$ and $(\clone_k^n)'$, of the $\rho_n$ and $\clone_k^n$ respectively, clearly satisfy the conditions to yield a $d$-ary cloning system on $(H_n)_{n\in\N}$. We call the $d$-ary cloning system $((H_n)_{n\in\N},(\rho_n')_{n\in\N},((\clone_k^n)')_{k\le n})$ a \emph{$d$-ary cloning subsystem} of $((G_n)_{n\in\N},(\rho_n)_{n\in\N},(\clone_k^n)_{k\le n})$.
\end{definition}

A $d$-ary cloning subsystem on a family $H_n \le G_n$ yields a Thompson-like subgroup $\Thomp{H_*}\le \Thomp{G_*}$. The fact that the $H_n$ are stable under applying the cloning maps ensures that $\Thomp{H_*}$ is closed under taking products.

\subsection{$d$-ary cloning systems for tree automorphism groups}\label{sec:aut_clone}

Now consider the family of groups $(S_n \wr \Aut(\tree_d))_{n\in\N}$. For notational convenience set
$$A_n\defeq S_n \wr \Aut(\tree_d)\text{.}$$
It will be helpful to be able to visualize elements of $A_n$. We can draw an element $\sigma(f_1,\dots,f_n)\in A_n$ by drawing $n$ arrows representing $\sigma$ (like in Figure~\ref{fig:symm_clone}), labeled by elements of $\Aut(\tree_d)$ at the bases of the arrows. See Figure~\ref{fig:A_n_element} for an example.

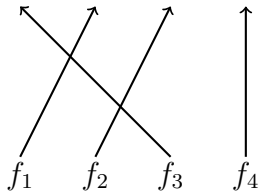
\begin{figure}[htb]
 \centering
 \begin{tikzpicture}[line width=0.8pt]
  \draw[->] (0,-2) -- (1,0); \draw[->] (1,-2) -- (2,0); \draw[->] (2,-2) -- (0,0); \draw[->] (3,-2) -- (3,0);
	\node at (0,-2.25) {$f_1$}; \node at (1,-2.25) {$f_2$}; \node at (2,-2.25) {$f_3$}; \node at (3,-2.25) {$f_4$};
 \end{tikzpicture}
 \caption{The element $(1~2~3)(f_1,f_2,f_3,f_4)$ of $A_4$.}
 \label{fig:A_n_element}
\end{figure}

Multiplication in the wreath product is mirrored in the pictures. To draw the product $\sigma(f_1,\dots,f_n)\tau(g_1,\dots,g_n)$, we stack the picture for $\sigma(f_1,\dots,f_n)$ on top of the picture for $\tau(g_1,\dots,g_n)$ and pull the $f_i$ down through $\tau$ to get the picture for $\sigma\tau(f_{\tau(1)}g_1,\dots,f_{\tau(n)}g_n)$. See Figure~\ref{fig:wreath_mult} for an example.

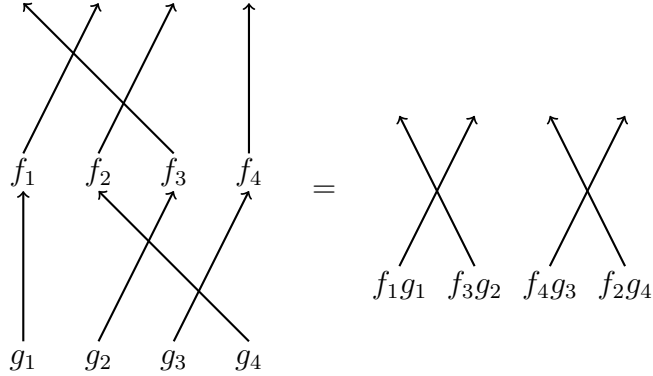
\begin{figure}[htb]
 \centering
 \begin{tikzpicture}[line width=0.8pt]
  \draw[->] (0,-2) -- (1,0); \draw[->] (1,-2) -- (2,0); \draw[->] (2,-2) -- (0,0); \draw[->] (3,-2) -- (3,0);
	\node at (0,-2.25) {$f_1$}; \node at (1,-2.25) {$f_2$}; \node at (2,-2.25) {$f_3$}; \node at (3,-2.25) {$f_4$};
	\begin{scope}[yshift=-2.5cm]
	 \draw[->] (0,-2) -- (0,0); \draw[->] (1,-2) -- (2,0); \draw[->] (2,-2) -- (3,0); \draw[->] (3,-2) -- (1,0);
	 \node at (0,-2.25) {$g_1$}; \node at (1,-2.25) {$g_2$}; \node at (2,-2.25) {$g_3$}; \node at (3,-2.25) {$g_4$};
	 \node at (4,0) {$=$};
	\end{scope}
	\begin{scope}[xshift=5cm,yshift=-1.5cm]
	 \draw[->] (0,-2) -- (1,0); \draw[->] (1,-2) -- (0,0); \draw[->] (2,-2) -- (3,0); \draw[->] (3,-2) -- (2,0);
	\node at (0,-2.25) {$f_1 g_1$}; \node at (1,-2.25) {$f_3 g_2$}; \node at (2,-2.25) {$f_4 g_3$}; \node at (3,-2.25) {$f_2 g_4$};
	\end{scope}
 \end{tikzpicture}
 \caption{A visualization of $(1~2~3)(f_1,f_2,f_3,f_4)(2~3~4)(g_1,g_2,g_3,g_4)=(1~2)(3~4)(f_1 g_1,f_3 g_2,f_4 g_3,f_2 g_4)$ in $A_4$.}
 \label{fig:wreath_mult}
\end{figure}
 
We now encode the family $(A_n)_{n\in\N}$ into a cloning system. For each $n\in\N$ define $\rho_n \colon A_n \to S_n$ via $\sigma(f_1,\dots,f_n) \mapsto \sigma$. For each $1\le k\le n$ define $\clone_k^n \colon A_n \to A_{n+d-1}$ via
$$\sigma(f_1,\dots,f_n) \mapsto (\sigma)\symmclone_k^n \rho^{(k)}(f_k)(f_1,\dots,f_{k-1},f_k^1,\dots,f_k^d,f_{k+1},\dots,f_n)\text{,}$$
where $f_k=\rho(f_k)(f_k^1,\dots,f_k^d)$ is the wreath recursion of $f_k$, and $\rho^{(k)}$ is the composition of $\rho \colon \Aut(\tree_d) \to S_d$ with the ($k$-dependent) monomorphism $S_d \to S_{n+d-1}$ induced by the injection $\{1,\dots,d\} \to \{1,\dots,n+d-1\}$ sending $j$ to $k+j-1$. Note that $\clone_k^n$ is injective, since $\symmclone_k^n$ is injective and $f_k$ is uniquely determined by its wreath recursion.

See Figure~\ref{fig:aaut_clone} for an example of how $\clone_k^n$ works.

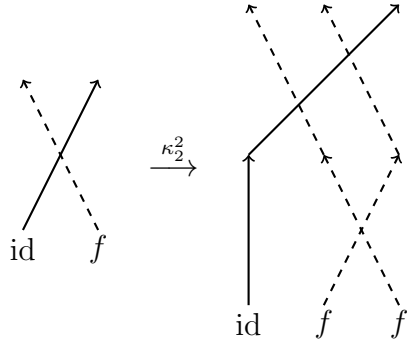
\begin{figure}[htb]
 \centering
 \begin{tikzpicture}[line width=0.8pt]
  \draw[->] (0,-2) -- (1,0); \draw[->,dashed] (1,-2) -- (0,0);
	\node at (0,-2.25) {$\id$}; \node at (1,-2.25) {$f$};
  \node at (2,-1) {$\stackrel{\clone_2^2}{\longrightarrow}$};
  \begin{scope}[xshift=3cm,yshift=1cm]
   \draw[->] (0,-2) -- (2,0); \draw (1,-2)[->,dashed] -- (0,0); \draw (2,-2)[->,dashed] -- (1,0);
	 \draw[->] (0,-4) -- (0,-2); \draw[->,dashed] (1,-4) -- (2,-2); \draw[->,dashed] (2,-4) -- (1,-2);
   \node at (0,-4.25) {$\id$}; \node at (1,-4.25) {$f$}; \node at (2,-4.25) {$f$};
  \end{scope}
 \end{tikzpicture}
 \caption{An example of $2$-ary cloning in the $A_n$. Here $f=(1~2)(f,f)\in\Aut(\tree_2)$ denotes the automorphism induced from the bijection $\{0,1\}\to\{0,1\}$ switching $0$ and $1$. The picture shows that $((1~2)(\id,f))\clone_2^2 = (1~3~2)(2~3)(\id,f,f)$. We use dashed lines to indicate the arrow getting cloned and the resulting arrows. To see where the $(1~3~2)(2~3)$ in the output comes from, note that $(1~2)\symmclone_2^2=(1~3~2)$ and $\rho^{(2)}(f)=(2~3)$.}
 \label{fig:aaut_clone}
\end{figure}

It is also helpful to observe that $\clone_1^1$ is nothing more than the usual wreath recursion. More precisely, $A_1=S_1\wr \Aut(\tree_d)=\Aut(\tree_d)$ and $(f)\clone_1^1=\rho(f)(f_1,\dots,f_d)$ for any $f\in A_1$.
 
\begin{proposition}\label{prop:aut_clone}
 $((A_n)_{n\in\N},(\rho_n)_{n\in\N},(\clone_k^n)_{k\le n})$ is a $d$-ary cloning system.
\end{proposition}

\begin{proof}
 For (C1), consider $f=\sigma(f_1,\dots,f_n)$ and $g=\tau(g_1,\dots,g_n)$ in $A_n$ and let $1\le k\le n$. Multiplication in the wreath product $A_n = S_n \wr \Aut(\tree_d)$ gives us $fg=\sigma\tau(f_{\tau(1)} g_1,\dots,f_{\tau(n)} g_n)$. The $k$th entry of this is $f_{\tau(k)} g_k$, and looking at the definition of $\clone_k^n$ we see we will need to know the wreath recursion of $f_{\tau(k)} g_k$. If we write the wreath recursions $f_{\tau(k)}=\rho(f_{\tau(k)})(f_{\tau(k)}^1,\dots,f_{\tau(k)}^d)$ and $g_k=\rho(g_k)(g_k^1,\dots,g_k^d)$, then
$$f_{\tau(k)} g_k=\rho(f_{\tau(k)})\rho(g_k)(f_{\tau(k)}^{\rho(g_k)(1)} g_k^1,\dots,f_{\tau(k)}^{\rho(g_k)(d)} g_k^d)\text{.}$$
We can now apply $\clone_k^n$ to $fg$ to get that the left hand side of (C1) is
\begin{align*}
(fg)\clone_k^n =& (\sigma\tau)\symmclone_k^n \rho^{(k)}(f_{\tau(k)} g_k)\\
~& (f_{\tau(1)}g_1,\dots,f_{\tau(k-1)}g_{k-1},f_{\tau(k)}^{\rho(g_k)(1)} g_k^1,\dots,f_{\tau(k)}^{\rho(g_k)(d)} g_k^d,f_{\tau(k+1)}g_{k+1},\dots,f_{\tau(n)}g_n)\text{.}
\end{align*}
Call $(\sigma\tau)\symmclone_k^n \rho^{(k)}(f_{\tau(k)} g_k)$ the ``permutation part'' of the left hand side, denoted $P_\ell$, and call the remaining expression the ``tuple part'' of the left hand side, denoted $T_\ell$. Now we move to the right hand side of (C1), where we will also isolate a permutation part $P_r$ and tuple part $T_r$. The right hand side is the product $(f)\clone_{\tau(k)}^n (g)\clone_k^n$, which we can write down knowing the wreath recursions for $f$ and $g$. We get
\begin{align*}
(f)\clone_{\tau(k)}^n (g)\clone_k^n
=& (\sigma)\symmclone_{\tau(k)}^n \rho^{(\tau(k))}(f_{\tau(k)}) (f_1,\dots,f_{{\tau(k)}-1},f_{\tau(k)}^1,\dots,f_{\tau(k)}^d,f_{{\tau(k)}+1},\dots,f_n)\\
~& (\tau)\symmclone_k^n \rho^{(k)}(g_k) (g_1,\dots,g_{k-1},g_k^1,\dots,g_k^d,g_{k+1},\dots,g_n)\text{.}
\end{align*}
Multiplication in $A_{n+d-1} = S_{n+d-1} \wr \Aut(\tree_d)$ tells us that this equals
\begin{align*}
(\sigma)\symmclone_{\tau(k)}^n \rho^{(\tau(k))}(f_{\tau(k)}) (\tau)\symmclone_k^n \rho^{(k)}(g_k)\\
~& (f_1,\dots,f_{{\tau(k)}-1},f_{\tau(k)}^1,\dots,f_{\tau(k)}^d,f_{{\tau(k)}+1},\dots,f_n)^{(\tau)\symmclone_k^n \rho^{(k)}(g_k)}\\
~& (g_1,\dots,g_{k-1},g_k^1,\dots,g_k^d,g_{k+1},\dots,g_n)
\end{align*}
so now we have our $P_r$ and $T_r$. We need to prove that $P_\ell=P_r$ and $T_\ell=T_r$, and then we will be done verifying (C1). To prove that $P_\ell=P_r$ we need to show that
$$(\sigma\tau)\symmclone_k^n \rho^{(k)}(f_{\tau(k)} g_k) = (\sigma)\symmclone_{\tau(k)}^n \rho^{(\tau(k))}(f_{\tau(k)}) (\tau)\symmclone_k^n \rho^{(k)}(g_k)\text{.}$$
We can expand $P_\ell$ using the fact that (C1) holds for the $d$-ary cloning maps $\symmclone_k^n$ for the symmetric groups and the fact that $\rho^{(k)}$ is a homomorphism, and get
$$P_\ell = (\sigma)\symmclone_{\tau(k)}^n (\tau)\symmclone_k^n \rho^{(k)}(f_{\tau(k)}) \rho^{(k)}(g_k)\text{.}$$
It now suffices to show that $(\tau)\symmclone_k^n \rho^{(k)}(f_{\tau(k)}) = \rho^{(\tau(k))}(f_{\tau(k)}) (\tau)\symmclone_k^n$. For $1\le i \le n+d-1$, if $i\ne k,\dots,k+d-1$ then $\rho^{(k)}(f_{\tau(k)})(i)=i$, so $(\tau)\symmclone_k^n \rho^{(k)}(f_{\tau(k)})(i) = (\tau)\symmclone_k^n (i)$. For such $i$, we have $(\tau)\symmclone_k^n (i) \ne \tau(k),\dots,\tau(k)+d-1$, so $\rho^{(\tau(k))}(f_{\tau(k)})$ fixes $(\tau)\symmclone_k^n (i)$ and we conclude that both sides of the equality send $i$ to $(\tau)\symmclone_k^n (i)$. Now suppose $i\in \{k,\dots,k+d-1\}$. Let $\phi_k \colon \{1,\dots,d\}\to\{k,\dots,k+d-1\}$ be the bijection sending $j$ to $j+k-1$, so $\rho^{(k)}(f_{\tau(k)})(i) = \phi_k \rho(f_{\tau(k)}) \phi_k^{-1}(i)$. We have
\begin{align*}
(\tau)\symmclone_k^n \rho^{(k)}(f_{\tau(k)})(i) &= (\tau)\symmclone_k^n \phi_k \rho(f_{\tau(k)})\phi_k^{-1}(i)\\
&= (\tau)\symmclone_k^n \phi_k \rho(f_{\tau(k)})(i-k+1)\\
&= (\tau)\symmclone_k^n \Big(\rho(f_{\tau(k)})(i-k+1) +k-1\Big)\\
&= \tau(k)+\Big(\rho(f_{\tau(k)})(i-k+1) +k-1\Big)-k\\
&= \rho(f_{\tau(k)})(i-k+1) + \tau(k)-1\\
&= \phi_{\tau(k)} \rho(f_{\tau(k)})(i-k+1)\\
&= \phi_{\tau(k)} \rho(f_{\tau(k)}) \phi_{\tau(k)}^{-1}(\tau(k)+i-k)\\
&= \rho^{(\tau(k))}(f_{\tau(k)}) (\tau(k)+i-k)\\
&= \rho^{(\tau(k))}(f_{\tau(k)}) (\tau)\symmclone_k^n (i)
\end{align*}

To prove that $T_\ell=T_r$ it suffices (after deleting the common $g_i$ factors) to prove that
\begin{align*}
~&(f_{\tau(1)},\dots,f_{\tau(k-1)},f_{\tau(k)}^{\rho(g_k)(1)},\dots,f_{\tau(k)}^{\rho(g_k)(d)},f_{\tau(k+1)},\dots,f_{\tau(n)})\\
&=(f_1,\dots,f_{{\tau(k)}-1},f_{\tau(k)}^1,\dots,f_{\tau(k)}^d,f_{{\tau(k)}+1},\dots,f_n)^{(\tau)\symmclone_k^n \rho^{(k)}(g_k)}\text{.}
\end{align*}
The action of $(\tau)\symmclone_k^n \rho^{(k)}(g_k) \in S_{n+d-1}$ on $(f_1,\dots,f_{{\tau(k)}-1},f_{\tau(k)}^1,\dots,f_{\tau(k)}^d,f_{{\tau(k)}+1},\dots,f_n)$ is given by permuting the entries. More precisely, this $(n+d-1)$-tuple is a function from $\{1,\dots,n+d-1\}$ to $\Aut(\tree_d)$ and the action of $S_{n+d-1}$ is via precomposition. Since the action of $S_{n+d-1}$ on $\{1,\dots,n+d-1\}$ is a left action, its action on the tuples is a right action (as indicated by the notation). When we apply $(\tau)\symmclone_k^n$ to the tuple we get $(f_{\tau(1)},\dots,f_{\tau(k-1)},f_{\tau(k)}^1,\dots,f_{\tau(k)}^d,f_{\tau(k+1)},\dots,f_{\tau(n)})$. Then when we follow this up with $\rho^{(k)}(g_k)$ we get $(f_{\tau(1)},\dots,f_{\tau(k-1)},f_{\tau(k)}^{\rho(g_k)(1)},\dots,f_{\tau(k)}^{\rho(g_k)(d)},f_{\tau(k+1)},\dots,f_{\tau(n)})$ as desired. This concludes the verification of (C1). Also see Figure~\ref{fig:C1_example} for a helpful picture of (C1) in action.

For (C2) we need to show that $\clone_\ell^n \circ \clone_k^{n+d-1} = \clone_k^n \circ \clone_{\ell+d-1}^{n+d-1}$ whenever $k<\ell$. Applying the left side function to $f=\sigma(f_1,\dots,f_n)\in A_n$ we obtain
\begin{align*}
~& \Big((\sigma)\symmclone_\ell^n \rho^{(\ell)}(f_\ell)(f_1,\dots,f_{\ell-1},f_\ell^1,\dots,f_\ell^d,f_{\ell+1},\dots,f_n)\Big)\clone_k^{n+d-1}\\
=& \Big((\sigma)\symmclone_\ell^n \rho^{(\ell)}(f_\ell)\Big)\symmclone_k^{n+d-1} \rho^{(k)}(f_k)\\
~& (f_1,\dots,f_{k-1},f_k^1,\dots,f_k^d,f_{k+1},\dots,f_{\ell-1},f_\ell^1,\dots,f_\ell^d,f_{\ell+1},\dots,f_n)\text{.}
\end{align*}
Since $k<\ell$ we have the following five useful facts:
\begin{enumerate}[label={(\arabic*)}, ref={\arabic*}, leftmargin=*]
 \item $\rho^{(\ell)}(f_\ell)(k) = k$ \label{item:k_fixed}
 \item $\Big(\rho^{(\ell)}(f_\ell)\Big)\symmclone_k^{n+d-1} = \rho^{(\ell+d-1)}(f_\ell)$ \label{item:k_clones}
 \item $\rho^{(k)}(f_k)\rho^{(\ell+d-1)}(f_\ell) = \rho^{(\ell+d-1)}(f_\ell)\rho^{(k)}(f_k)$ \label{item:rho_comm}
 \item $\rho^{(k)}(f_k)=\Big(\rho^{(k)}(f_k)\Big)\symmclone_{\ell+d-1}^{n+d-1}$ \label{item:k_cloned}
 \item $\rho^{(k)}(f_k)(\ell+d-1) = \ell+d-1$ \label{item:k_fixes}\text{.}
\end{enumerate}
Using these, together with (C1) and (C2) for the $d$-ary cloning maps $\symmclone_k^n$ on the symmetric groups, which we will call (S1) and (S2) here, gives us
\begin{align*}
\Big((\sigma)\symmclone_\ell^n \rho^{(\ell)}(f_\ell)\Big)\symmclone_k^{n+d-1} \rho^{(k)}(f_k) &\stackrel{\text{(S1)},\eqref{item:k_fixed}}{=} \Big((\sigma)\symmclone_\ell^n\Big)\symmclone_k^{n+d-1} \Big(\rho^{(\ell)}(f_\ell)\Big)\symmclone_k^{n+d-1} \rho^{(k)}(f_k)\\
&\stackrel{\text{(S2)},\eqref{item:k_clones}}{=} \Big((\sigma)\symmclone_k^n\Big)\symmclone_{\ell+d-1}^{n+d-1} \rho^{(\ell+d-1)}(f_\ell) \rho^{(k)}(f_k)\\
&\stackrel{\eqref{item:rho_comm}}{=} \Big((\sigma)\symmclone_k^n\Big)\symmclone_{\ell+d-1}^{n+d-1} \rho^{(k)}(f_k) \rho^{(\ell+d-1)}(f_\ell)\\
&\stackrel{\eqref{item:k_cloned}}{=} \Big((\sigma)\symmclone_k^n\Big)\symmclone_{\ell+d-1}^{n+d-1} \Big(\rho^{(k)}(f_k)\Big)\symmclone_{\ell+d-1}^{n+d-1} \rho^{(\ell+d-1)}(f_\ell)\\
&\stackrel{\text{(S1)},\eqref{item:k_fixes}}{=} \Big((\sigma)\symmclone_k^n \rho^{(k)}(f_k)\Big)\symmclone_{\ell+d-1}^{n+d-1} \rho^{(\ell+d-1)}(f_\ell)\text{.}
\end{align*}
We conclude that indeed
\begin{align*}
(f)\clone_\ell^n \circ \clone_k^{n+d-1} =& \Big((\sigma)\symmclone_\ell^n \rho^{(\ell)}(f_\ell)\Big)\symmclone_k^{n+d-1} \rho^{(k)}(f_k)\\
~& (f_1,\dots,f_{k-1},f_k^1,\dots,f_k^d,f_{k+1},\dots,f_{\ell-1},f_\ell^1,\dots,f_\ell^d,f_{\ell+1},\dots,f_n)\\
=& \Big((\sigma)\symmclone_k^n \rho^{(k)}(f_k)\Big)\symmclone_{\ell+d-1}^{n+d-1} \rho^{(\ell+d-1)}(f_\ell)\\
~& (f_1,\dots,f_{k-1},f_k^1,\dots,f_k^d,f_{k+1},\dots,f_{\ell-1},f_\ell^1,\dots,f_\ell^d,f_{\ell+1},\dots,f_n)\\
=& (f)\clone_k^n \circ \clone_{\ell+d-1}^{n+d-1}\text{.}
\end{align*}

Finally we prove (C3). This holds almost trivially: For $f=\sigma(f_1,\dots,f_n)$ and $i\ne k,\dots,k+d-1$ we have
$$\rho_{n+d-1}((f)\clone_k^n)(i) = (\sigma)\symmclone_k^n \rho^{(k)}(f_k)(i) = (\sigma)\symmclone_k^n(i) = (\rho_n(f))\symmclone_k^n(i)$$
since the support of $\rho^{(k)}(f_k)$ is $\{k,\dots,k+d-1\}$.
\end{proof}

\begin{figure}[htb]
 \centering
 \begin{tikzpicture}[line width=0.8pt]
  \draw[->] (0,-2) -- (0,0); \draw[->,dashed] (1,-2) -- (2,0); \draw[->,dashed] (2,-2) -- (1,0);
	\node at (0,-2.25) {$f$}; \node at (1,-2.25) {$f$}; \node at (2,-2.25) {$f$};
  \begin{scope}[xshift=5cm,yshift=2cm]
   \draw[->] (2,-2) -- (0,0); \draw (0,-2)[->,dashed] -- (1,0); \draw (1,-2)[->,dashed] -- (2,0);
   \node at (0,-2.25) {$\id$}; \node at (1,-2.25) {$\id$}; \node at (2,-2.25) {$f$};
	 \begin{scope}[yshift=-2.5cm]
	  \draw[->] (0,-2) -- (2,0); \draw (1,-2)[->,dashed] -- (0,0); \draw (2,-2)[->,dashed] -- (1,0);
	  \draw[->] (0,-4) -- (0,-2); \draw[->,dashed] (1,-4) -- (2,-2); \draw[->,dashed] (2,-4) -- (1,-2);
    \node at (0,-4.25) {$\id$}; \node at (1,-4.25) {$f$}; \node at (2,-4.25) {$f$};
	 \end{scope}
  \end{scope}
 \end{tikzpicture}
 \caption{Returning to the example in Figure~\ref{fig:aaut_clone}, we will check that the ``cloning a product'' axiom (C1) holds when applying $\clone_2^2$ to the product $((1~2)(\id,f))((1~2)(\id,f))$. Note that $((1~2)(\id,f))((1~2)(\id,f))=(f,f)$. The left picture here is of the left side of (C1), which is $(f,f)\clone_2^2$, and the right picture is of the right side of (C1), which is $((1~2)(\id,f))\clone_1^2 ((1~2)(\id,f))\clone_2^2$. From the pictures, it is clear that these are indeed the same.}
 \label{fig:C1_example}
\end{figure}
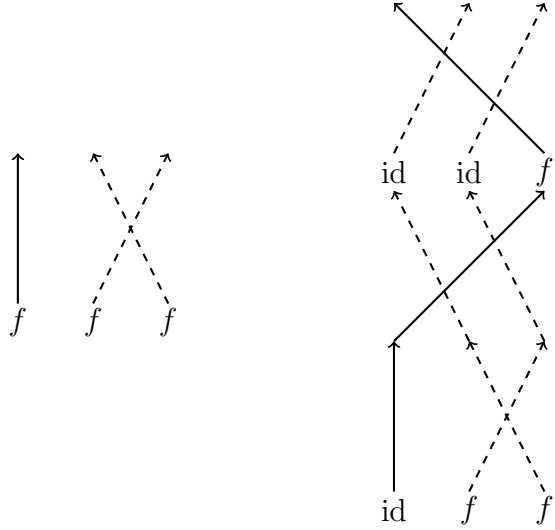

The $d$-ary cloning system on $(A_n)_{n\in\N}$ yields a Thompson group $\Thomp{A_*}$. Note that the Higman--Thompson group $V_d$ embeds into $\Thomp{A_*}$ as the subgroup of elements of the form $[T_-,\sigma\vv{\id},T_+]$. Also, for $T$ any finite rooted $d$-ary tree, say with $n$ leaves, $A_n$ embeds into $\Thomp{A_*}$ as the subgroup $A(T) \defeq \{[T,f,T]\mid f\in A_n\}$, by \cite[Observation~3.1]{witzel18} (or really its natural $d$-ary analog). For example $\Aut(\tree_d)$ itself embeds as $A(1_1)$ via $f\mapsto [1_1,f,1_1]$, where $1_1$ the trivial tree. We will sometimes equivocate between $\Aut(\tree_d)$ and $A(1_1)$ without comment.

We are almost ready to prove that $\Thomp{A_*}\cong \AAut(\tree_d)$ (Theorem~\ref{thrm:aaut_clone}), but first need a technical lemma. This lemma should also reveal to the reader where our definition of the cloning maps $\clone_k^n$ came from.

\begin{lemma}\label{lem:clone_once}
 Let $((T_-,\sigma(f_1,\dots,f_n),T_+)) \in \AAut(\tree_d)$. Choose $1\le k\le n$, and let $U_+$ be the finite rooted $d$-ary tree obtained by adding a $d$-caret to the $k$th leaf of $T_+$, so $U_+$ has $n+d-1$ leaves. Then $((T_-,\sigma(f_1,\dots,f_n),T_+)) = ((U_-,\tau(g_1,\dots,g_{n+d-1}),U_+))$ if and only if $U_-$ equals $T_-$ with a $d$-caret added to its $\sigma(k)$th leaf and $\tau(g_1,\dots,g_{n+d-1}) = (\sigma(f_1,\dots,f_n))\clone_k^n$.
\end{lemma}

\begin{proof}
 First suppose that $U_-$ equals $T_-$ with a $d$-caret added to its $\sigma(k)$th leaf and that $\tau(g_1,\dots,g_{n+d-1}) = (\sigma(f_1,\dots,f_n))\clone_k^n$, and we want to show that
$$((T_-,\sigma(f_1,\dots,f_n),T_+)) = ((U_-,(\sigma(f_1,\dots,f_n))\clone_k^n,U_+))\text{.}$$
By Citation~\ref{cit:same_aaut} applied to this special case, it suffices to show that when we restrict $\sigma(f_1,\dots,f_n)$ to $U_+^c$ it maps bijectively to $U_-^c$ and coincides with $(\sigma(f_1,\dots,f_n))\clone_k^n$. That it restricts to a bijection $U_+^c \to U_-^c$ is just because $\sigma(f_1,\dots,f_n)$ takes the $k$th tree in $T_+^c$ to the $\sigma(k)$th tree in $T_-^c$. Now we need to understand what this restriction looks like. Applying $\sigma(f_1,\dots,f_n)$ to $T_+^c$ amounts to first applying $f_i$ to the $i$th tree, for each $1\le i\le n$, and then applying $\sigma$ to shuffle the trees. When we remove the $d$-caret on the $k$th root from $T_+^c$ to obtain $U_+^c$, this becomes: first apply $f_i$ to the $i$th tree, for each $1\le i<k$ and $k+d-1<i\le n$, then apply $f_k^i$ to the $(i+k-1)$st tree for each $1\le i\le d$, where $f_k=\rho(f_k)(f_k^1,\dots,f_k^d)$ is the wreath recursion of $f_k$, and then shuffle the trees, first by shuffling trees $k$ through $k+d-1$ according to $\rho(f_k)$ and then by shuffling all the trees by $(\sigma)\symmclone_k^n$. But this exactly describes $(\sigma(f_1,\dots,f_n))\clone_k^n$.

 Now suppose $((T_-,\sigma(f_1,\dots,f_n),T_+)) = ((U_-,\tau(g_1,\dots,g_{n+d-1}),U_+))$. Let $U_-^?$ be $T_-$ with a $d$-caret added to its $\sigma(k)$th leaf, so by the previous paragraph we know
$$((T_-,\sigma(f_1,\dots,f_n),T_+))=((U_-^?,(\sigma(f_1,\dots,f_n))\clone_k^n,U_+))$$
and hence
$$((U_-,\tau(g_1,\dots,g_{n+d-1}),U_+))=((U_-^?,(\sigma(f_1,\dots,f_n))\clone_k^n,U_+))\text{.}$$
We need to show that $U_-=U_-^?$ and $\tau(g_1,\dots,g_{n+d-1})=(\sigma(f_1,\dots,f_n))\clone_k^n$. Choose a finite $V_+\supseteq U_+$ such that the restrictions of $\tau(g_1,\dots,g_{n+d-1})$ and $(\sigma(f_1,\dots,f_n))\clone_k^n$ to $V_+^c$ coincide, so composing one with the inverse of the other is the identity on $V_+^c$, hence is necessarily the identity on $U_+^c$, so indeed $\tau(g_1,\dots,g_{n+d-1})=(\sigma(f_1,\dots,f_n))\clone_k^n$. Then it is immediate that $U_-=U_-^?$ since the former is the complement of the image of $U_+^c$ under $\tau(g_1,\dots,g_{n+d-1})$ and the latter is the complement of the image of $U_+^c$ under $(\sigma(f_1,\dots,f_n))\clone_k^n$.
\end{proof}

We can now prove our main result connecting $d$-ary cloning systems and almost-automorphisms.

\begin{theorem}[$\AAut(\tree_d)$ from cloning]\label{thrm:aaut_clone}
 We have $\Thomp{A_*}\cong \AAut(\tree_d)$.
\end{theorem}

\begin{proof}
 Define $\Phi \colon \Thomp{A_*} \to \AAut(\tree_d)$ via $[T_-,f,T_+]\mapsto ((T_-,f,T_+))$ and $\Psi \colon \AAut(\tree_d) \to \Thomp{A_*}$ via $((T_-,f,T_+))\mapsto [T_-,f,T_+]$. We need to prove that both of these are well defined homomorphisms, and then since they are obviously inverses they will be isomorphisms.

 First we claim $\Phi$ is well defined. Suppose $[T_-,f,T_+]=[U_-,g,U_+]$ and we need to show that $((T_-,f,T_+))=((U_-,g,U_+))$. Since we know we can get from $(T_-,f,T_+)$ to $(U_-,g,U_+)$ via a sequence of reductions and expansions, it suffices to prove that $((T_-,f,T_+))=((U_-,g,U_+))$ holds whenever $(U_-,g,U_+)$ is an expansion of $(T_-,f,T_+)$. Say $U_+$ is $T_+$ with a $d$-caret added to the $k$th leaf, so $g=(f)\clone_k^n$ and $U_-$ is $T_-$ with a $d$-caret added to the $\sigma(k)$th leaf (here $f=\sigma(f_1,\dots,f_n)$). Now Lemma~\ref{lem:clone_once} says $((T_-,f,T_+))=((U_-,g,U_+))$.

 Next we claim $\Psi$ is well defined. Suppose $((T_-,f,T_+))=((U_-,g,U_+))$ and we need to show that $[T_-,f,T_+]=[U_-,g,U_+]$. First note that if $U_+$ is $T_+$ with a $d$-caret added to the $k$th leaf, then Lemma~\ref{lem:clone_once} says that $U_-$ is $T_-$ with a $d$-caret added to the $\sigma(k)$th leaf (where again $f=\sigma(f_1,\dots,f_n)$) and $g=(f)\clone_k^n$. In particular $[T_-,f,T_+]=[U_-,g,U_+]$. Now, returning to the general case, Citation~\ref{cit:same_aaut} says there exist finite rooted complete subtrees $V_-$ and $V_+$ of $\tree_d$, with $T_-,U_-\subseteq V_-$ and $T_+,U_+\subseteq V_+$, such that $f$ and $g$ restrict to bijections $V_+^c \to V_-^c$ and these bijections coincide. In particular $((T_-,f,T_+))=((V_-,h,V_+))=((U_-,g,U_+))$ where $h$ is the common restriction of $f$ and $g$ to $V_+^c \to V_-^c$. We can build up from $T_+$ to $V_+$ adding a $d$-caret at a time, and the previous discussion then gives us $[T_-,f,T_+]=[V_-,h,V_+]$. Then a parallel argument shows $[V_-,h,V_+]=[U_-,g,U_+]$, so $[T_-,f,T_+]=[U_-,g,U_+]$ as desired.

 Lastly we claim $\Phi$ is a homomorphism (and then $\Psi=\Phi^{-1}$ is too, for trivial reasons). Let $[T_-,f,T_+],[U_-,g,U_+]\in \Thomp{A_*}$. After possibly expanding representative triples, we can assume $T_+=U_-$. In particular $[T_-,f,T_+][U_-,g,U_+]=[T_-,fg,U_+]$. Since $((T_-,f,T_+))((U_-,g,U_+))=((T_-,fg,U_+))$ is obviously true whenever $T_+=U_-$, we have
\begin{align*}
 \Phi([T_-,f,T_+][U_-,g,U_+])&=\Phi([T_-,fg,U_+])=((T_-,fg,U_+))\\
 &=((T_-,f,T_+))((U_-,g,U_+))=\Phi([T_-,f,T_+])\Phi([U_-,g,U_+])\text{.}
\end{align*}
\end{proof}

Now let $G\le \Aut(\tree_d)$ be self-similar. Since $G$ is self-similar we have $\clone_k^n(S_n\wr G) \subseteq S_{n+d-1}\wr G$. Hence our cloning system on $(A_n)_{n\in\N}$ restricts to a cloning subsystem on $(S_n\wr G)_{n\in\N}$, and we get a Thompson-like group $\Thomp{S_*\wr G}$.

\begin{corollary}[R\"over--Nekrashevych from cloning]\label{cor:nek_clone}
 For any self-similar $G\le\Aut(\tree_d)$, we have $\Thomp{S_*\wr G}\cong V_d(G)$.
\end{corollary}

\begin{proof}
 Let $\Phi \colon \Thomp{A_*} \to \AAut(\tree_d)$ be the isomorphism $[T_-,f,T_+] \mapsto ((T_-,f,T_+))$ from the proof of Theorem~\ref{thrm:aaut_clone}. We have that $[T_-,f,T_+]\in \Thomp{S_*\wr G}$ if and only if we can choose $f\in S_n\wr G$ (maybe after performing some expansions). Then $\Phi([T_-,f,T_+])=((T_-,f,T_+)) \in V_d(G)$. Hence $\Phi$ takes $\Thomp{S_*\wr G}$ into $V_d(G)$, and it is onto by a parallel argument involving $\Psi=\Phi^{-1}$.
\end{proof}


\subsection{Strand diagrams}\label{sec:strand}

There is a convenient way to visualize elements of $\Thomp{A_*}$ via versions of strand diagrams, like for elements of the classical and braided Thompson groups (see, e.g., \cite{brady08}). Let $[T_-,\sigma(f_1,\dots,f_n),T_+]\in\Thomp{A_*}$, so $T_-$ and $T_+$ have $n$ leaves. Draw $T_-$ right side up, i.e., with its root at the top, and $T_+$ below it and upside down, i.e., with its root at the bottom. In between the leaves of $T_-$ and $T_+$ draw $\sigma(f_1,\dots,f_n)$, like in Figure~\ref{fig:A_n_element}, connecting arrows to leaves as appropriate. We get a \emph{strand diagram} for $[T_-,\sigma(f_1,\dots,f_n),T_+]$. The top (range) tree, $T_-$ is a sequence of \emph{splits} and the bottom (domain) tree is a sequence of \emph{merges}, with $\sigma(f_1,\dots,f_n)$ in the middle. Two such pictures represent the same element of $\Thomp{A_*}$ if we can get from one to the other via a sequence of moves analogous to reductions and expansions. For example the expansion move $[T_-,\sigma(f_1,\dots,f_n),T_+]=[T_-',(\sigma(f_1,\dots,f_n))\clone_k^n,T_+']$ can be visualized by introducing a split-merge at the $k$th leaf of $T_+$ and pulling the split up through $\sigma(f_1,\dots,f_n)$. The split lands at the $\sigma(k)$th leaf of $T_-$ and $\sigma(f_1,\dots,f_n)$ gets cloned to become $(\sigma(f_1,\dots,f_n))\clone_k^n$. See Figure~\ref{fig:TA_element}.

\begin{figure}[htb]
 \centering
 \begin{tikzpicture}[line width=0.8pt]
  \draw (0,0) -- (0.5,0.5) -- (1,0);
	\draw[->] (0,-1) -- (1,0); \draw[->,dashed] (1,-1) -- (0,0);
	\node at (0,-1.25) {$\id$}; \node at (1,-1.25) {$f$};
	\draw (0,-1.5) -- (0.5,-2) -- (1,-1.5);
	\node at (-0.5,0.25) {$T_-$}; \node at (-0.5,-1.75) {$T_+$};
	\begin{scope}[xshift=3cm]
	 \draw (0.5,0.5) -- (1,1) -- (2,0); \draw[dashed] (0,0) -- (0.5,0.5) -- (1,0);
	 \draw[->] (0,-1) -- (2,0); \draw[->,dashed] (1,-1) -- (0,0); \draw[->,dashed] (2,-1) -- (1,0);
	 \draw[->] (0,-2) -- (0,-1); \draw[->,dashed] (1,-2) -- (2,-1); \draw[->,dashed] (2,-2) -- (1,-1);
	 \node at (0,-2.25) {$\id$}; \node at (1,-2.25) {$f$}; \node at (2,-2.25) {$f$};
	 \draw (0,-2.5) -- (1,-3.5) -- (1.5,-3); \draw[dashed] (1,-2.5) -- (1.5,-3) -- (2,-2.5);
	\end{scope}
 \end{tikzpicture}
 \caption{The left picture represents an element $[T_-,\sigma(f_1,f_2),T_+]$ of $\protect\Thomp{S_*\wr \Aut(\tree_d)}$. Here $d=2$, $T_-$ and $T_+$ are each a single $2$-caret, $\sigma=(1~2)$, and $(f_1,f_2)=(\id,f)$ where $f=(1~2)(f,f)$ is as in Figure~\ref{fig:aaut_clone}. The right picture is an expansion of the left picture using $k=2$. The arrows relevant to the cloning are dashed.}
 \label{fig:TA_element}
\end{figure}
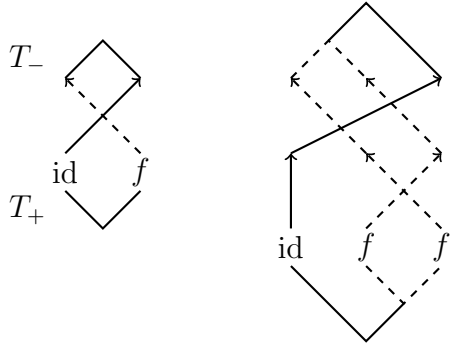

Multiplication in $\Thomp{A_*}$ is mirrored by stacking strand diagrams, with the convention that left-to-right corresponds to top-to-bottom; see Figure~\ref{fig:strand_mult} for an example.

\begin{figure}[htb]
 \centering
 \begin{tikzpicture}[line width=0.8pt]
  \draw (0,0) -- (0.5,0.5) -- (1,0);
	\draw[->] (0,-1) -- (1,0); \draw[->] (1,-1) -- (0,0);
	\node at (0,-1.25) {$\id$}; \node at (1,-1.25) {$f$};
	\draw (0,-1.5) -- (0.5,-2) -- (1,-1.5)   (0.5,-2) -- (0.5,-2.25);
	\node at (-0.5,0.25) {$T_-$}; \node at (-0.5,-1.75) {$T_+$};
  \draw (-0.5,-3.25) -- (0.5,-2.25) -- (1.5,-3.25)   (0.5,-3.25) -- (1,-2.75);
	\node at (-0.5,-3.5) {$f$}; \node at (0.5,-3.5) {$f$}; \node at (1.5,-3.5) {$\id$};
	\draw (-0.5,-3.75) -- (0.5,-4.75) -- (1.5,-3.75)   (0.5,-3.75) -- (0,-4.25);
	\node at (-0.5,-2.5) {$U_-$}; \node at (-0.5,-4.5) {$U_+$};
	\node at (2,-2.25) {$=$};
	
	\begin{scope}[xshift=4cm,yshift=-1cm]
	 \draw (0,0) -- (0.5,0.5) -- (1,0);
	 \draw[->] (0,-1) -- (1,0); \draw[->] (1,-1) -- (0,0);
	 \node at (0,-1.25) {$\id$}; \node at (1,-1.25) {$f$};
	 \draw (0,-1.5) -- (-0.5,-2)   (0.5,-2) -- (1,-1.5) -- (1.5,-2);
	 \node at (-0.5,-2.25) {$f$}; \node at (0.5,-2.25) {$f$}; \node at (1.5,-2.25) {$\id$};
	 \draw (-0.5,-2.5) -- (0.5,-3.5) -- (1.5,-2.5)   (0.5,-2.5) -- (0,-3);
	 \node at (2.25,-1.25) {$=$};
	\end{scope}
	
	\begin{scope}[xshift=8cm,yshift=-1cm]
	 \draw (-0.5,0) -- (0.5,1) -- (1.5,0)   (0,0.5) -- (0.5,0);
	 \draw[->] (-0.5,-2) -- (1.5,0); \draw[->] (1.5,-2) -- (-0.5,0); \draw[->] (1.5,-1) -- (0.5,0);
	 \draw[->] (0.5,-2) -- (1.5,-1);
	 \node at (-0.5,-2.25) {$f$}; \node at (0.5,-2.25) {$f^2$}; \node at (1.5,-2.25) {$f$};
	 \draw (-0.5,-2.5) -- (0.5,-3.5) -- (1.5,-2.5)   (0.5,-2.5) -- (0,-3);
	 \node at (2.25,-1.25) {$=$};
	\end{scope}
	
	\begin{scope}[xshift=12cm,yshift=-1.5cm]
	 \draw (-0.5,0) -- (0.5,1) -- (1.5,0)   (0,0.5) -- (0.5,0);
	 \draw[->] (-0.5,-1) -- (1.5,0); \draw[->] (0.5,-1) -- (0.5,0); \draw[->] (1.5,-1) -- (-0.5,0);
	 \node at (-0.5,-1.25) {$f$}; \node at (0.5,-1.25) {$\id$}; \node at (1.5,-1.25) {$f$};
	 \draw (-0.5,-1.5) -- (0.5,-2.5) -- (1.5,-1.5)   (0.5,-1.5) -- (0,-2);
	 \node at (1.75,0.5) {$V_-$}; \node at (1.75,-2) {$V_+$};
	\end{scope}
 \end{tikzpicture}
 \caption{The multiplication $[T_-,(1~2)(\id,f),T_+][U_-,(f,f,\id),U_+]=[V_-,(1~3)(f,\id,f),V_+]$ in $\protect\Thomp{S_*\wr \Aut(\tree_d)}$, where the trees are all as drawn. Here $f=(1~2)(f,f)$ is as in Figure~\ref{fig:TA_element}, so in particular $f^2=\id$.}
 \label{fig:strand_mult}
\end{figure}
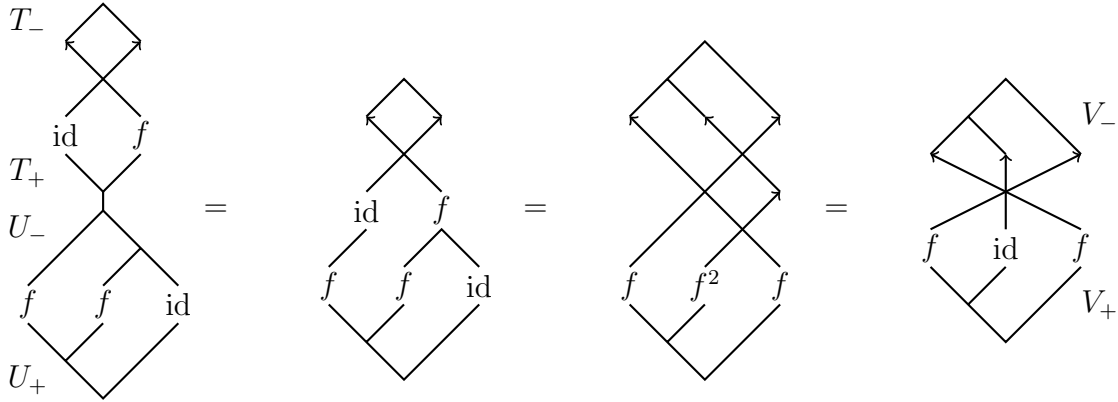

We will not make strand diagrams overly formal here, since we are just using them as a helpful visualization, but a formal theory could be worked out following the classical and braided versions; see \cite{belk14,brady08,witzel18}.

These strand diagrams also have an interpretation when viewing such elements as almost-automorphisms of $\tree_d$. Roughly, view $\tree_d$ with its root at the top and stack $[T_-,f,T_+]$ on top of it. The domain tree $T_+$ (of merges) cancels away the copy of $T_+$ in $\tree_d$, then $f$ acts on the resulting forest of $n$ subtrees, and finally the range tree $T_-$ (of splits) closes the forest back up into $\tree_d$.

\begin{remark}\label{rmk:G_F_infty}
 Now that we know $V_d(G)\cong \Thomp{S_*\wr G}$, one could develop a $d$-ary analog of \cite[Proposition~5.9]{witzel18} to quickly prove that if $G$ is of type $\F_\infty$ then so is $V_d(G)$. Rather than devote time here to developing the $d$-ary analog of this specific machinery (which would be worth doing in the future), we will prove our results for more general $G$ in the coming sections, with this result as a trivial consequence (Corollary~\ref{cor:Finfty_inherited}).
\end{remark}


\section{Morse theory and a first complex}\label{sec:large_cpx}

For the rest of the paper, let $G\le\Aut(\tree_d)$ be self-similar and fix a subgroup $H\le G$. In this section we recall some background on finiteness properties and discrete Morse theory, and then construct a simplicial complex $|P_H^1|$ on which $V_d(G)$ acts with stabilizers closely related to $H$. After imposing some conditions on $H$, the action of $V_d(G)$ on $|P_H^1|$ will be particularly nice. In Section~\ref{sec:stein} we will retract $|P_H^1|$ down to an even nicer complex, the so-called $H$-Stein--Farley complex $\Stein{S_*\wr G}_H$, and impose still more conditions on $H$, which will allow us to prove $V_d(G)$ is of type $\F_\infty$ in certain cases. All of this is heavily inspired by Belk and Matucci's work in \cite{belk16} proving that the R\"over group is of type $\F_\infty$.

\subsection{Morse theory}\label{sec:morse}

A typical tool for proving that a group is type $\F_\infty$ is to couple Brown's Finiteness Criterion \cite{brown87} with Bestvina--Brady Morse theory \cite{bestvina97}. Before stating the result (Lemma~\ref{lem:morse}) we need some definitions. Let $X$ be an affine cell complex as in \cite{bestvina97}, for example a simplicial complex or a cube complex. Let $h\colon X \to \R$ be a map that restricts to a non-constant affine map on each positive dimensional cell of $X$, and such that $h(X^{(0)})$ is discrete in $\R$. We call $h$ a \emph{Morse function}. For each $x\in X^{(0)}$ define the \emph{descending link} $\dlk x$ of $x$ to be the subcomplex of $\lk x$ spanned by those vertices $y\in \lk x$ such that $h(y)< h(x)$. For each $q\in\R$ define the \emph{sublevel set} $X^{h\le q}$ to be the subcomplex of $X$ consisting of those cells contained in $h^{-1}((-\infty,q])$.

We can now state our combination of Brown's Criterion and the Morse Lemma. It is almost identical to \cite[Theorem~5.7]{belk16} except we do not insist on stabilizers being finite, merely of type $\F_\infty$. Since the assumption on stabilizers is just to ensure we can apply Brown's Criterion, and since Brown's Criterion goes through even with $\F_\infty$ stabilizers, this lemma is immediate.

\begin{lemma}\label{lem:morse}
 Let $\Gamma$ be a group acting cellularly on a contractible affine cell complex $X$. Let $h\colon X\to \R$ be a $\Gamma$-equivariant Morse function. Then $\Gamma$ is of type $\F_\infty$ provided that the following conditions are satisfied:
 \begin{enumerate}
  \item For each $q\in\R$ the sublevel set $X^{h\le q}$ is $\Gamma$-cocompact.
	\item Each cell stabilizer is of type $\F_\infty$.
	\item For each $n\in\N$ there exists $t\in\R$ such that for all $x\in X^{(0)}$ with $h(x)\ge t$ we have that $\dlk x$ is $n$-connected.
 \end{enumerate}
\end{lemma}

\subsection{Groupoid and poset}\label{sec:gpoid_poset}

The goal of this subsection is to construct a simplicial complex on which $V_d(G)$ acts nicely. Much of what we do here is similar to what Belk and Matucci did in \cite{belk16} for the R\"over group, though more in the cloning system language from \cite{witzel18}.

First we define an important groupoid $\Groupoid{S_*\wr G}$ related to the group $\Thomp{S_*\wr G}$. Consider triples of the form $(F_-,f,F_+)$ for $F_-$ and $F_+$ finite rooted $d$-ary forests with the same number of leaves, say $n$, and $f$ an element of $S_n\wr G$. Here by \emph{rooted $d$-ary forest} we always mean an ordered tuple of rooted $d$-ary trees, and the roots of these trees are called the \emph{roots} of the forest. Note that if a finite rooted $d$-ary forest has $n$ leaves then it necessarily has at most $n$ roots. Also note that in the triple $(F_-,f,F_+)$ we do not require $F_-$ and $F_+$ to have the same number of roots, only the same number of leaves. We have the same notion of reduction and expansion as with triples $(T_-,f,T_+)$, namely if $F_+'$ is $F_+$ with a $d$-caret added to its $k$th leaf and $F_-'$ is $F_-$ with a $d$-caret added to its $\rho_n(f)(k)$th leaf, then $(F_-',(f)\clone_k^n,F_+')$ is an \emph{expansion} of $(F_-,f,F_+)$. The reverse of an expansion is a \emph{reduction}. Two triples are considered equivalent if we can get from one to the other via a sequence of reductions and expansions, and we write $[F_-,f,F_+]$ for the equivalence class of $(F_-,f,F_+)$. Note that the number of roots of each forest is invariant under this equivalence relation.

\begin{definition}[Heads, feet]
 The \emph{heads} of $[F_-,f,F_+]$ are the roots of $F_-$ and the \emph{feet} are the roots of $F_+$.
\end{definition}

These equivalence classes of triples form a groupoid $\Groupoid{S_*\wr G}$ (roughly, a groupoid is like a group except only certain pairs of elements can be multiplied). For $x,y\in\Groupoid{S_*\wr G}$ the product $xy$ exists if and only if the number of feet of $x$ equals the number of heads of $y$. In this (and only this) case there exists a sequence of expansions of $x$ and $y$ respectively yielding $[F,f,E]$ and $[E,g,D]$ for some $F$, $E$, $D$, $f$ and $g$. Then the product is $xy=[F,f,E][E,g,D]\defeq [F,fg,D]$. This is easily checked to be a groupoid operation, for the same reasons $\Thomp{S_*\wr G}$ is a group. In fact $\Thomp{S_*\wr G}$ is precisely the subgroup of $\Groupoid{S_*\wr G}$ consisting of elements with $1$ head and $1$ foot. The identity elements of $\Groupoid{S_*\wr G}$ are those of the form $[1_n,\vv{\id},1_n]$ for $1_n$ the forest consisting only of $n$ roots (which are thus also leaves). We will also sometimes denote $[1_n,\vv{\id},1_n]$ by $1_n$, and it will always be clear whether this symbol represents a forest or a groupoid element.

\begin{remark}\label{rmk:gpoid_cantor}
 In exactly the same way that the group $\Thomp{S_*\wr \Aut(\tree_d)}$ is isomorphic to the group of almost-automorphisms of $\tree_d$, the groupoid $\Groupoid{S_*\wr \Aut(\tree_d)}$ is isomorphic to the groupoid of ``almost-isomorphisms'' between forests of finitely many copies of $\tree_d$. That is, an element $[F_-,\sigma\vv{f},F_+]$ with $m$ heads and $n$ feet takes $n$ copies of $\tree_d$, chops away $F_+$, performs $\sigma\vv{f}$ on the resulting trees, and glues them together with $F_-$ into $m$ copies of $\tree_d$. In this way, we can view $[F_-,\sigma\vv{f},F_+]$ as a homeomorphism from $n$ copies of $\partial\tree_d$ to $m$ copies of $\partial\tree_d$.
\end{remark}

For $n,m\in \N$ let $P_n^m$ be the subset of $\Groupoid{S_*\wr G}$ defined by
$$P_n^m\defeq \{x\mid x\text{ has }m\text{ heads and }n\text{ feet}\}\text{.}$$
Also set
$$P^m\defeq \bigcup_{n\in\N} P_n^m \text{ and } P_n\defeq \bigcup_{m\in\N} P_n^m \text{.}$$
Later we will be especially interested in $P^1$, since $\Thomp{S_*\wr G}=P_1^1$ acts from the left on $P^1$.

Recall that we have fixed a subgroup $H\le G$. For each $n\in\N$ let $[1_n,S_n \wr H,1_n]$ be the group $\{[1_n,w,1_n]\mid w\in S_n \wr H\}$ inside $\Groupoid{S_*\wr G}$, so it makes sense to multiply an element of $P_n$ on the right by an element of $[1_n,S_n \wr H,1_n]$. For $x\in P_n$ write $[x]_H$ for the coset $x[1_n,S_n \wr H,1_n]$. For notational ease, if $x=[F_-,f,F_+]$ we will write $[F_-,f,F_+]_H$ for $[x]_H$ rather than $[[F_-,f,F_+]]_H$. Define $P_H$ to be
$$P_H\defeq \bigcup_{n\in\N} P_n/[1_n,S_n \wr H,1_n] \text{.}$$
Also write $P_H^m$ for the subset of $P_H$ consisting of those $[x]_H$ such that $x$ has $m$ heads. Write $\sim_H$ for the equivalence relation on $\Groupoid{S_*\wr G}$ yielding $P_H$, so $P_H=\Groupoid{S_*\wr G}/\sim_H$.

We want to put a poset structure on $P_H$. First we need to define a \emph{splitting}. Let $\vertwedge_k^n$ be the $d$-ary forest obtained from $1_n$ by adding a $d$-caret to its $k$th leaf.

\begin{definition}[Splitting, length]\label{def:splitting}
 An element $\Lambda\in\Groupoid{S_*\wr G}$ of the form
$$[1_n,\sigma_0\vv{h}_0,1_n][\vertwedge_{k_1}^n,\sigma_1\vv{h}_1,1_{n+d-1}]\cdots[\vertwedge_{k_r}^{n+(r-1)(d-1)},\sigma_r\vv{h}_r,1_{n+r(d-1)}]$$
for some $r$ and some choice of $\sigma_i\vv{h}_i\in S_{n+i(d-1)} \wr H$ ($0\le i\le r$) and $1\le k_i\le n+(i-1)(d-1)$ ($1\le i\le r$) is called a \emph{splitting}. Note that $\Lambda$'s status as a splitting depends on the choice of $H$, but since we have fixed $H$ there will not be any terminological ambiguities. We call $r$ the \emph{length} of $\Lambda$. If $r=0$ we call $\Lambda$ a \emph{trivial splitting}. The length $r$ is well defined since $n$ and $n+r(d-1)$ are invariants of $\Lambda$, namely they are its number of heads and feet. Length is also invariant under $\sim_H$.
\end{definition}

Now define a partial order $\le$ on $P_H$ by declaring that $[x]_H\le [y]_H$ whenever $x^{-1}y$ is a splitting. This is obviously reflexive, is antisymmetric since any non-trivial splitting strictly increases the number of feet, and is transitive since a product of splittings is a splitting. Also, if $[x]_H=[x']_H$ and $[y]_H=[y']_H$ then $x^{-1}y$ is a splitting if and only if $(x')^{-1}y'$ is a splitting, so the definition of $\le$ does not require us to locate any special representatives of elements of $P_H$.

Elements of $\Groupoid{S_*\wr G}$ can be drawn similarly to elements of $\Thomp{S_*\wr G}$. Figure~\ref{fig:splitting} gives a $2$-ary example of an element that is a length-$2$ splitting.

\begin{figure}[htb]
 \centering
 \begin{tikzpicture}[line width=0.8pt]
  \draw[->] (0,0) -- (2,1);   \draw[->] (2,0) -- (0,1);
	\node at (0,-.25) {$h_1$}; \node at (2,-.25) {$h_2$};
	\draw (-.5,-1) -- (0,-.5) -- (.5,-1)   (1.5,-1) -- (2,-.5) -- (2.5,-1);
	\draw[->] (-.5,-2) -- (-.5,-1); \draw[->] (.5,-2) -- (1.5,-1); \draw[->](1.5,-2) -- (2.5,-1); \draw[->](2.5,-2) -- (.5,-1);
  \node at (-.5,-2.25) {$h_3$}; \node at (.5,-2.25) {$h_4$}; \node at (1.5,-2.25) {$h_5$}; \node at (2.5,-2.25) {$h_6$};
 \end{tikzpicture}
 \caption{The splitting $[1_2,(1~2)(h_1,h_2),1_2][\protect\vertwedge_1^2,\protect\vv{\id},1_3][\protect\vertwedge_3^3,(2~3~4)(h_3,h_4,h_5,h_6),1_4]$ in $\protect\Groupoid{S_*\wr \Aut(\tree_d)}$. (Here $d=2$.)}
 \label{fig:splitting}
\end{figure}

Note that number of heads and number of feet are invariant under the equivalence relation coming from $H$, and the product $x\Lambda$ only makes sense if the number of feet of $x$ equals the number of heads of $\Lambda$. If we want to indicate that the product $x\Lambda$ makes sense, we may call $\Lambda$ a \emph{splitting of $x$}. We may also refer to $[x\Lambda]_H$ as a \emph{splitting of $[x]_H$}.

As a remark, using \cite[Definition~3.1]{belk16} in the R\"over group case, Belk and Matucci would call $[y]_H$ an ``expansion'' of $[x]_H$ if $[x]_H\le [y]_H$, and reserve the term ``splitting'' for those where $\Lambda=x^{-1}y$ has length $1$. We will just use the term splitting, and will say ``length-$1$ splitting'' when we want to only consider those.

There is a natural Morse function $\feet$, given by counting feet, on the geometric realization $|P_H^m|$, namely for $[x]_H \in P_H^m$ with $n$ feet, define
$$\feet([x]_H)\defeq n \text{,}$$
and then extend $\feet$ affinely to the simplices of $|P_H^m|$. Since $[x]_H<[y]_H$ in $P_H^m$ implies $\feet([x]_H)<\feet([y]_H)$ we know $\feet$ is non-constant on positive-dimensional simplices, and the image of $\feet$ in $\R$ is discrete (in fact it equals $\{m+(d-1)q\mid q\in \N\cup\{0\}\}$), so $\feet$ is indeed a Morse function on $|P_H^m|$.

At this point we begin gradually imposing assumptions on $H$. The first assumption ensures that for each $[x]_H\in P_H^m$ there are only finitely many $[y]_H>[x]_H$ with a given $\feet$-value $\feet([y]_H)$. This will be crucial when we need to deal with issues of cocompactness and finiteness properties of stabilizers. We call this assumption ``$A$-coarsely self-similar'':

\begin{definition}[$A$-coarsely self-similar, artifact]\label{def:Acss}
 Let $A$ be a finite subset of $V_d\cap\Aut(\tree_d)$. Call a subgroup $C\le \Aut(\tree_d)$ \emph{$A$-coarsely self-similar} if for every $c\in C$, given the wreath recursion $c=\sigma(c_1,\dots,c_d)$ we have that for each $1\le i\le d$ either $c_i\in C$ or else $c_i\in A$. The elements of $A$ will be called \emph{artifacts}.
\end{definition}

\begin{lemma}[Key finiteness lemma]\label{lem:fin_splits}
 If $H$ is $A$-coarsely self-similar for some finite $A$, then for each $[x]_H\in P_H^m$ and each $q\in\N$ there are only finitely many $[y]_H\in P_H^m$ with $[x]_H<[y]_H$ and $\feet([y]_H)=q$.
\end{lemma}

\begin{proof}
 Letting $n\defeq \feet([x]_H)$, the relevant $[y]_H$ are precisely those of the form $[x\Lambda]_H$ for $\Lambda$ a splitting satisfying
$$[\Lambda]_H=[1_n,\sigma_0\vv{h}_0,1_n][\vertwedge_{k_1}^n,\sigma_1\vv{h}_1,1_{n+d-1}]\cdots[\vertwedge_{k_r}^{n+(r-1)(d-1)},\sigma_r\vv{h}_r,1_q][1_q,S_q \wr H,1_q]$$
with length $r=\frac{q-n}{d-1}$, and some choice of $\sigma_i\vv{h}_i\in S_{n+i(d-1)} \wr H$ and $1\le k_i\le n+(i-1)(d-1)$. (If $q-n$ is not divisible by $d-1$ then no such $y$ exist, so we can assume $q-n$ is divisible by $d-1$.)

 We need to show that as the $\sigma_i$, $\vv{h}_i$ and $k_i$ range over all possibilities, the set of groupoid cosets of the above form is finite. Note that for arbitrary $p$, $k$ and $\sigma\vv{h}\in S_p \wr H$ we have $[1_p,\sigma\vv{h},1_p][\vertwedge_k^p,\vv{\id},1_{p+d-1}]=[\vertwedge_{\sigma(k)}^p,(\sigma\vv{h})\clone_k^p,1_{p+d-1}]$. Since $H$ is $A$-coarsely self-similar, the tuple in the output of $(\sigma\vv{h})\clone_k^p$ has entries from $H\cup A$. Hence (assuming without loss of generality that $\id\in A$) we can rewrite
$$[1_n,\sigma_0\vv{h}_0,1_n][\vertwedge_{k_1}^n,\sigma_1\vv{h}_1,1_{n+d-1}]\cdots[\vertwedge_{k_r}^{n+(r-1)(d-1)},\sigma_r\vv{h}_r,1_q][1_q,S_q \wr H,1_q]$$
as
$$[1_n,\sigma_0,1_n][\vertwedge_{\ell_1}^n,\sigma_1\vv{a}_1,1_{n+d-1}]\cdots[\vertwedge_{\ell_r}^{n+(r-1)(d-1)},\sigma_r\vv{a}_r,1_q][1_q,\sigma\vv{h},1_q][1_q,S_q \wr H,1_q]$$
for some $1\le \ell_i\le n+(i-1)(d-1)$, $\vv{a}_i\in A^{n+i(d-1)}$ and $\sigma\vv{h}\in S_{n+r(d-1)}\wr H$. Intuitively, we have pushed all the things from $H$ to the right. Now observe that $[1_q,\sigma\vv{h},1_q][1_q,S_q \wr H,1_q]=[1_q,S_q \wr H,1_q]$, so our coset is of the form
$$[1_n,\sigma_0,1_n][\vertwedge_{\ell_1}^n,\sigma_1\vv{a}_1,1_{n+d-1}]\cdots[\vertwedge_{\ell_r}^{n+(r-1)(d-1)},\sigma_r\vv{a}_r,1_q][1_q,S_q \wr H,1_q]\text{,}$$
with the only variable parameters being the $\sigma_i$, $\vv{a}_i$ and $\ell_i$, and indeed there are only finitely many possibilities.
\end{proof}

The action of $V_d(G)=\Thomp{S_*\wr G}=P_1^1$ on $P^1$ from the left respects $\sim_H$, and the poset structure on $P^1$ and $P_H^1$, so $V_d(G)$ acts on the simplicial complex $|P_H^1|$. We now use Lemma~\ref{lem:fin_splits} to show that this action is ``nice'' assuming $H$ is $A$-coarsely self-similar.

\begin{lemma}[Stabilizers]\label{lem:stabs}
 Suppose $H$ is $A$-coarsely self-similar for some finite $A$. Let $[x]_H$ be a $0$-simplex in $|P_H^1|$, say $\feet([x]_H)=n$. Then the stabilizer in $V_d(G)$ of $[x]_H$ is isomorphic to $S_n \wr H$. The stabilizer of any $k$-simplex in $|P_H^1|$ with vertices $[x]_H<[x_1]_H<\cdots<[x_k]_H$ is a finite index subgroup of $\Stab_{V_d(G)}([x]_H)$, hence a finite index subgroup of $S_n \wr H$.
\end{lemma}

\begin{proof}
 An element $[U_-,f,U_+] \in V_d(G)$ stabilizes $[x]_H$ if and only if $[U_-,f,U_+]x[1_n,S_n \wr H,1_n] = x[1_n,S_n \wr H,1_n]$, if and only if there exists $[1_n,\sigma\vv{h},1_n]\in [1_n,S_n \wr H,1_n]$ such that $[U_-,f,U_+]x=x[1_n,\sigma\vv{h},1_n]$, if and only if $x^{-1}[U_-,f,U_+]x=[1_n,\sigma\vv{h},1_n]$. Hence there is an injective map from the stabilizer $\Stab_{V_d(G)}([x]_H)$ to $[1_n,S_n \wr H,1_n]$ via
$$[U_-,f,U_+] \mapsto x^{-1}[U_-,f,U_+]x\text{.}$$
Moreover, $x[1_n,\sigma\vv{h},1_n]x^{-1}$ lies in $\Stab_{V_d(G)}([x]_H)$ for any $[1_n,\sigma\vv{h},1_n]\in [1_n,S_n\wr H,1_n]$, so this map is also surjective and indeed $\Stab_{V_d(G)}([x]_H)$ is isomorphic to $S_n \wr H$.

 Now consider $[x]_H<[x_1]_H<\cdots<[x_k]_H$, and call the simplex $\Delta$. Since the action of $V_d(G)$ on $|P_H^1|$ preserves the Morse function $\feet$, and since $\feet([x]_H)<\cdots<\feet([x_k]_H)$, we know that $\Stab_{V_d(G)}(\Delta)=\Stab_{V_d(G)}([x]_H)\cap\Stab_{V_d(G)}([x_1]_H)\cap\cdots\cap \Stab_{V_d(G)}([x_k]_H)$. Hence the result will follow if we can prove that $\Stab_{V_d(G)}([x]_H)\cap \Stab_{V_d(G)}([y]_H)$ has finite index in $\Stab_{V_d(G)}([x]_H)$ whenever $[x]_H<[y]_H$.

Say $[x]_H<[y]_H=x\Lambda[1_q,S_q \wr H,1_q]$ is arbitrary. All the vertices $[y']_H$ in the orbit of $[y]_H$ under $\Stab_{V_d(G)}([x]_H)$ have $\feet([y']_H)=q$, and Lemma~\ref{lem:fin_splits} says only finitely many such $[y']_H$ exist. Since this orbit is finite, $\Stab_{V_d(G)}([x]_H)\cap \Stab_{V_d(G)}([y]_H) = \Stab_{\Stab_{V_d(G)}([x]_H)}([y]_H)$ has finite index in $\Stab_{V_d(G)}([x]_H)$.
\end{proof}

\begin{corollary}\label{cor:stabs}
 Suppose $H$ is $A$-coarsely self-similar for some finite $A$, and is of type $\F_\infty$. Then every stabilizer in $V_d(G)$ of a simplex in $|P_H^1|$ is of type $\F_\infty$.
\end{corollary}

\begin{proof}
 By Lemma~\ref{lem:stabs} every stabilizer is a finite index subgroup of some $S_n \wr H = S_n \ltimes (H\times\cdots\times H)$. Since $H$ is of type $\F_\infty$ so is $H\times\cdots\times H$, and since a group is of type $\F_\infty$ if and only if its finite index subgroups are of type $\F_\infty$, the result is immediate.
\end{proof}

Another application of Lemma~\ref{lem:fin_splits} is the following cocompactness result.

\begin{lemma}[Cocompact]\label{lem:cocpt}
 Suppose $H$ is $A$-coarsely self-similar for some finite $A$. For each $q\in\N$, the sublevel set $|P_H^1|^{\feet\le q}$ is $V_d(G)$-cocompact.
\end{lemma}

\begin{proof}
 First we claim $|P_H^1|^{\feet\le q}$ has finitely many vertex orbits. Indeed, if vertices $[T_-,f,F_+]_H$ and $[U_-,g,E_+]_H$ have the same $\feet$-value, so $F_+$ and $E_+$ have the same number of roots, then the product $[U_-,g,E_+][F_+,f^{-1},T_-]$ makes sense in $\Groupoid{S_*\wr G}$, and is an element of $\Thomp{S_*\wr G}=V_d(G)$. It takes $[T_-,f,F_+]_H$ to $[U_-,g,E_+]_H$, so indeed there is only one vertex orbit for each $\feet$-value. Then since $H$ is $A$-coarsely self-similar, Lemma~\ref{lem:fin_splits} implies that there are only finitely many simplex orbits in $|P_H^1|^{\feet\le q}$.
\end{proof}

We would like the posets $P_H^m$ to be directed (meaning any two elements have a common upper bound), since then their geometric realizations $|P_H^m|$ will be contractible. In particular then $|P_H^1|$ will be a contractible complex on which $V_d(G)$ acts. This brings us to our second assumption we will impose on $H$.

\begin{definition}[Nuclear subgroup]\label{def:nuclear}
 View the vertex set of $\tree_d$ as $X^*$ for some $d$-letter alphabet $X$. We will call $H$ \emph{nuclear in $G$} if for all $g\in G$ there exists a finite rooted complete subtree $T$ of $\tree_d$ such that for each $u\in X^*$ with $u$ the address of a leaf of $T$, the state $g_u$ lies in $H$.
\end{definition}

\begin{lemma}\label{lem:poset_cible}
 If $H$ is nuclear in $G$ then $P_H^m$ is directed, so $|P_H^m|$ is contractible.
\end{lemma}

\begin{proof}
 We must show that any two elements of $P_H^m$ have a common upper bound. It suffices to prove that each element of $P_H^m$ has an upper bound of the form $[F,\vv{\id},1_n]_H$, since any two such elements with the same number of heads $m$ clearly have a common upper bound, namely $[F,\vv{\id},1_q]_H,[E,\vv{\id},1_r]_H\le [F\cup E,\vv{\id},1_p]$ (for appropriate $p$). Let $[F_-,f,F_+]_H$ be an arbitrary element of $P_H^m$, say $F_-$ and $F_+$ have $q$ leaves and $f=\sigma(f_1,\dots,f_q)$. Pass to the upper bound $[F_-,f,1_q]_H$. Since $H$ is nuclear in $G$ we can choose a finite rooted $d$-ary forest $E=(T_1,\dots,T_q)$, say with $n$ total leaves, such that for each $1\le i\le q$ and each $u\in X^*$ that is the address of a leaf of $T_i$, the state $(f_i)_u$ lies in $H$. In particular if we expand $1_q$ to $E$, then $f$ gets replaced by its image under a sequence of appropriate cloning maps to some $f'\in S_n \wr H$, and $F_-$ gets replaced by some $F$ with $n$ leaves, leaving us with $[F_-,f,1_q]_H=[F,f',E]_H$. Pass to the upper bound $[F,f',1_n]_H$. Since $f'\in S_n \wr H$, this equals $[F,\vv{\id},1_n]_H$ and we are done.
\end{proof}

\begin{remark}
 Note that any $G$ is nuclear in itself, so the $|P_G^m|$ are always contractible. However, if $G$ is not of type $\F_\infty$ then neither are the stabilizers in $V_d(G)$ of simplices in $|P_G^1|$, which makes $|P_G^1|$ rather useless if we want to prove that $V_d(G)$ is of type $\F_\infty$. It is possible though that the fact the $|P_G^m|$ is always contractible could be useful in the future.
\end{remark}

To recap, if $H$ is $A$-coarsely self-similar for some finite set of artifacts $A$, and nuclear in $G$, then $|P_H^1|$ is a contractible complex on which $V_d(G)$ acts with stabilizers of type $\F_\infty$, with a $V_d(G)$-equivariant Morse function $\feet$, such that the filtration $|P_H^1|^{\feet\le q}$ is cocompact. In particular all the conditions in Lemma~\ref{lem:morse} are satisfied except the one about descending links. As often happens (e.g., in \cite{belk16,bux16,farley03,farley15,fluch13,martinez-perez16}) the descending links in $|P_H^1|$ are too large to get a good handle on, so in the next section we will retract $|P_H^1|$ to a more manageable complex and analyze descending links there. In the course of this we will need to impose one last assumption on $H$ (see Definition~\ref{def:orderly}).


\section{The $H$-Stein--Farley complex}\label{sec:stein}

The construction of the $H$-Stein--Farley complex here will be rather similar to the construction in \cite{belk16} for the special case of the R\"over group of what Belk and Matucci call the Stein complex, though some details will be more complicated, due to our degree of generality. Again, in general our language will be more in line with the world of cloning systems and \cite{witzel18}.

\subsection{Constructing the complex}\label{sec:construct_stein}

It will be convenient to introduce some notation, namely the direct sum notation Belk and Matucci use in \cite{belk16}. To formally define this in our framework we need some setup. If $F_1,\dots,F_n$ are finite rooted $d$-ary forests, say the trees of $F_i$ are $T_i^1,\dots,T_i^{r_i}$ from left to right, let $F_1\oplus\cdots\oplus F_n$ be the forest whose trees are, from left to right, $T_1^1,\dots,T_1^{r_1},T_2^1,\dots,T_2^{r_2},\dots,T_n^1,\dots,T_n^{r_n}$. For tuples $(f_1^1,\dots,f_1^{r_1}),\dots,(f_r^1,\dots,f_r^{r_n})$ of elements of $G$ let $(f_1^1,\dots,f_1^{r_1})\oplus\cdots\oplus(f_r^1,\dots,f_r^{r_n})$ be the concatenation $(f_1^1,\dots,f_r^{r_n})$. For $\sigma_1\in S_{r_1},\dots,\sigma_n\in S_{r_n}$ let $\sigma_1\oplus\cdots\oplus\sigma_n\in S_{r_1+\cdots+r_n}$ be the permutation defined as follows: for $k\in\{1,\dots,r_1+\cdots+r_n\}$ choose the unique $i$ such that $r_1+\cdots+r_i<k\le r_1+\cdots+r_{i+1}$ (with the expression $k<r_1+\cdots+r_{n+1}$ understood to always hold) and declare that $\sigma_1\oplus\cdots\oplus\sigma_n$ takes $k$ to $\sigma_i(k-r_1-\cdots-r_i)+r_1+\cdots+r_i$. Now finally for $[F_1^-,\sigma_1\vv{f}_1,F_1^+],\dots,[F_n^-,\sigma_n\vv{f}_n,F_n^+]$ elements of $\Groupoid{S_*\wr G}$ define $[F_1^-,\sigma_1\vv{f}_1,F_1^+]\oplus\cdots\oplus[F_n^-,\sigma_n\vv{f}_n,F_n^+]$ to be
$$[F_1^-\oplus\cdots\oplus F_n^-,(\sigma_1\oplus\cdots\oplus\sigma_n)(\vv{f}_1\oplus\cdots\oplus\vv{f}_n),F_1^+\oplus\cdots\oplus F_n^+]\text{.}$$

This direct sum notation is useful for characterizing splittings.

\begin{lemma}\label{lem:decomp_up_bd}
 For $x_1,\dots,x_n\in\Groupoid{S_*\wr G}$, the upper bounds of $[x_1\oplus\cdots\oplus x_n]_H$ in $P_H$ are precisely the elements of the form $[y_1\oplus\cdots\oplus y_n]_H$ such that $[x_i]_H\le [y_i]_H$ for all $1\le i\le n$. In particular the upper bounds of an element of the form $[x(x_1\oplus\cdots\oplus x_n)]_H$ are precisely the elements $[x(y_1\oplus\cdots\oplus y_n)]_H$ such that $[x_i]_H\le [y_i]_H$ for all $1\le i\le n$.
\end{lemma}

\begin{proof}
 If $[x_i]_H\le [y_i]_H$ then $y_i$ is obtained from $x_i$ via right multiplication by a splitting, and a direct sum of splittings is itself a splitting, so it is clear that every such $[y_1\oplus\cdots\oplus y_n]_H$ is an upper bound of $[x_1\oplus\cdots\oplus x_n]_H$. Conversely, any upper bound of $[x_1\oplus\cdots\oplus x_n]_H$ is obtained by right multiplying $x_1\oplus\cdots\oplus x_n$ by a splitting and then passing to the $\sim_H$-equivalence class, but any splitting is $\sim_H$-equivalent to a direct sum of $1$-head splittings, each of which serves to pass to an upper bound of a single $[x_i]_H$. Hence every upper bound is of the desired form.

 The second claim follows simply because the left action of $\Groupoid{S_*\wr G}$ on itself preserves $\sim_H$-equivalence classes and respects the poset relation $\le$.
\end{proof}

\begin{corollary}\label{cor:decomp_split}
 Any $m$-head splitting $\Lambda$ satisfies $[\Lambda]_H=[\Lambda_1\oplus\cdots\oplus\Lambda_m]_H$ for some $1$-head splittings $\Lambda_1,\dots,\Lambda_m$.
\end{corollary}

\begin{proof}
 We have $[1_m]_H\le[\Lambda]_H$, $1_m=1_1\oplus\cdots\oplus1_1$ with $m$ summands, and $[1_1]_H\le[x]_H$ if and only if $x$ is a $1$-head splitting, so this follows from Lemma~\ref{lem:decomp_up_bd}.
\end{proof}

See Figure~\ref{fig:direct_decomp} for an example using the splitting from Figure~\ref{fig:splitting}.

\begin{figure}[htb]
 \centering
 \begin{tikzpicture}[line width=0.8pt]
  \draw[->] (0,0) -- (0,.5);   \draw[->] (2,0) -- (2,.5);
	\node at (0,-.25) {$h_2$}; \node at (2,-.25) {$h_1$};
	\draw (-.5,-1) -- (0,-.5) -- (.5,-1)   (1.5,-1) -- (2,-.5) -- (2.5,-1);
 \end{tikzpicture}
 \caption{The splitting from Figure~\ref{fig:splitting} is equivalent under $\sim_H$ to this splitting of the form $\Lambda_1\oplus \Lambda_2$, where $\Lambda_1=[1_1,h_2,1_1][\protect\vertwedge_1^1,(\id,\id),1_2]$ and $\Lambda_2=[1_1,h_1,1_1][\protect\vertwedge_1^1,(\id,\id),1_2]$.}
 \label{fig:direct_decomp}
\end{figure}
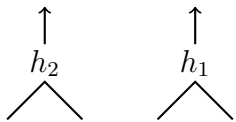

For the rest of the section we will assume $H$ is $A$-coarsely self-similar for some finite set of artifacts $A$, and nuclear in $G$. Note that $A$-coarsely self-similar implies $(A\setminus H)$-coarsely self-similar, and also for $A\subseteq B$ it is clear that $A$-coarsely self-similar implies $B$-coarsely self-similar, so without loss of generality $\id\in A$ and $H\cap A=\{\id\}$. Also, $A$-coarsely self-similar implies $(A\cap G)$-coarsely self-similar (since $G$ is self-similar), so without loss of generality we can assume $A\subseteq G$. We will make these assumptions on $A$ throughout this section. We will also assume $H$ and $A$ satisfy the following addition assumption.

\begin{definition}[Orderly]\label{def:orderly}
 Let $A$ be a finite subset of $V_d\cap\Aut(\tree_d)$. Call $(H,A)$ \emph{orderly} if the following conditions hold:
\begin{enumerate}
 \item The wreath recursion of any $a\in A$ is $a=\rho(a)(\id,\dots,\id)$,
 \item The wreath recursion of any $h\in H$ is $h=\rho(h)(h_1,\dots,h_d)$ for $h_{\rho(h)(1)}\in A\cup H$ and $h_i\in H$ for all $i\ne \rho(h)(1)$,
 \item If $a,b\in A$ with $a\ne b$ then $\rho(a)(1)\ne\rho(b)(1)$.
\end{enumerate}
\end{definition}

This convoluted definition is the most general one that we found to get all the coming results to work, but the working example to keep in mind is the following.

\begin{example}\label{ex:orderly}
 Let $a=[\vertwedge_1^1,(1\cdots d),\vertwedge_1^1]$ and let $A=\{a,a^2,\dots,a^d\}$. Suppose $H$ is such that the wreath recursion of any $h\in H$ is of the form $h=(a^i,h_2,\dots,h_d)$ for some $i$ and some $h_2,\dots,h_d\in H$. Then $(H,A)$ is orderly by construction.
\end{example}

In general note that if $(H,A)$ is orderly then for any $h\in H$ we have $[1_1,h,1_1][\vertwedge_1^1,\vv{\id},1_d]_H=[\vertwedge_1^1,(a,\id,\dots,\id),1_d]_H$ for some $a\in A$. The fact that the artifact must appear in the first entry is important, and is the motivation behind the second condition in Definition~\ref{def:orderly}. In general we will sometimes leave out trivial factors and use superscripts like $~^{(k)}$ to indicate what belongs in which entry, so for example $[\vertwedge_1^1,(a,\id,\dots,\id),1_d]_H$ could be denoted $[\vertwedge_1^1,a^{(1)},1_d]_H$.

The key result that requires $(H,A)$ to be orderly is the following.

\begin{lemma}[Least upper bound of artifacts]\label{lem:perm_lubs}
 For $a_1,\dots,a_r\in A$, the elements $[a_i]_H$ of $P_H^1$ have a unique least upper bound in $P_H^1$. More precisely, if $r\ge 2$ then $[\vertwedge_1^1,\vv{\id},1_d]_H$ is the least upper bound of the $[a_i]_H$.
\end{lemma}

\begin{proof}
 If $r=1$ then $[a_1]_H$ is its own least upper bound, so assume $r\ge2$. Since $(H,A)$ is orderly, for each $1\le i\le r$ we have $[a_i [\vertwedge_1^1,\vv{\id},1_d]]_H=[\vertwedge_1^1,\vv{\id},1_d]_H$, so this is in fact an upper bound of each $[a_i]_H$. We need to show that every common upper bound of $[a_1]_H,\dots,[a_r]_H$ is an upper bound of $[\vertwedge_1^1,\vv{\id},1_d]_H$. It is clear that if our claim holds in the $r=2$ case then it holds for all $r\ge2$, so we can assume $r=2$.

 First note that any strict upper bound of $[a_1]_H$ is an upper bound of some length-$1$ splitting of $[a_1]_H$, i.e., an element of the form $[a_1 h [\vertwedge_1^1,\vv{\id},1_d]]_H$ for some $h\in H$. Since $(H,A)$ is orderly, such an element equals $[a_1[\vertwedge_1^1,a^{(1)},1_d]]_H$ for some $a\in A$. As an element of $\Thomp{S_*\wr G}$ we have $a_1=[\vertwedge_1^1,\rho(a_1),\vertwedge_1^1]$, and hence $[a_1[\vertwedge_1^1,a^{(1)},1_d]]_H = [\vertwedge_1^1,\rho(a_1)a^{(1)},1_d]_H = [\vertwedge_1^1,a^{(\rho(a_1)(1))},1_d]_H$. By Lemma~\ref{lem:decomp_up_bd}, any upper bound of this is of the form $[\vertwedge_1^1,\vv{\id},1_d][\Lambda_1\oplus\cdots\oplus\Lambda_d]_H$ for some $1$-head splittings $\Lambda_i$ such that $[1_1]\le [\Lambda_i]_H$ for all $i\ne \rho(a_1)(1)$ and $[a]_H\le [\Lambda_{\rho(a_1)(1)}]_H$.

Doing a similar procedure for $a_2$ we get that any strict upper bound of $[a_2]_H$ is an upper bound of $[\vertwedge_1^1,b^{(\rho(a_2)(1))}),1_d]_H$ for some $b\in A$. Such an upper bound must be of the form $[\vertwedge_1^1,\vv{\id},1_d][\Lambda_1\oplus\cdots\oplus\Lambda_d]_H$ for some $1$-head splittings $\Lambda_i$ such that $[1_1]\le [\Lambda_i]_H$ for all $i\ne \rho(a_2)(1)$ and $[b]_H\le [\Lambda_{\rho(a_2)(1)}]_H$. If this is also an upper bound of $[a_1]_H$ then either $[1_1]_H\le[\Lambda_i]_H$ for all $i$, in which case it is an upper bound of $[\vertwedge_1^1,\vv{\id},1_d]_H$ and we are done, or else $\rho(a_1)(1)=\rho(a_2)(1)$. But this is ruled out in the definition of orderly, since $a_1\ne a_2$.
\end{proof}

\begin{remark}
 If $(H,A)$ is not orderly then Lemma~\ref{lem:perm_lubs} might not hold, i.e., minimal upper bounds might fail to be unique. For example, in the Gupta--Sidki group for $d=3$, generated by $a=(1~2~3)$ and $b=(a,a^{-1},b)$, the artifact $a$ can appear in either the first or second spot, since $b^{-1}=(a^{-1},a,b^{-1})$. This shows $(\langle b\rangle,\langle a\rangle)$ is not orderly, and in fact $[1_1]_{\langle b\rangle}$ and $[a]_{\langle b\rangle}$ do not have a unique least upper bound in $P_{\langle b\rangle}^1$: both $[\vertwedge_1^1,\vv{\id},1_3]_{\langle b\rangle}$ and $[\vertwedge_1^1,(\id,a,\id),1_3][\vertwedge_3^3,\vv{\id},1_5][\vertwedge_1^5,\vv{\id},1_7]_{\langle b\rangle}$ are minimal common upper bounds of $[1_1]_{\langle b\rangle}$ and $[a]_{\langle b\rangle}$.
\end{remark}

Next we show that if $(H,A)$ is orderly then any finite collection of length-$1$ splittings of a given $[x]_H$ have a unique least upper bound.

\begin{proposition}\label{prop:lubs}
 Any finite collection of length-$1$ splittings of some $[x]_H$ have a unique least upper bound in $P_H$.
\end{proposition}

\begin{proof}
 Since the action of $\Groupoid{S_*\wr G}$ from the left preserves the poset structure, we can assume $x=1_m$ (where $m$ is the number of feet of the original $x$). Any length-$1$ splitting of $[1_m]_H$ is of the form $[1_m,\sigma\vv{h},1_m][\vertwedge_\ell^m,\vv{\id},1_{m+d-1}]_H$ for some $\sigma\vv{h}\in S_m\wr H$ and some $\ell$, and since $H$ is $A$-coarsely self-similar and $(H,A)$ is orderly this equals $[\vertwedge_k^m,a^{(k)},1_{m+d-1}]_H$ for $k=\sigma(\ell)$ and some $a\in A$ (recall that $a^{(k)}$ indicates that $a$ appears in the $(k)$th entry of the tuple and the other entries are trivial). Note that the $k$th entry corresponds to the first leaf of the $d$-caret in $\vertwedge_k^m$.

 Now suppose $[\vertwedge_{k_j}^m,(a_j)^{(k_j)},1_{m+d-1}]_H$, $1\le j\le r$ are a collection of length-$1$ splittings of $[1_m]_H$. Without loss of generality $k_1\le\cdots\le k_r$. For any string of indices $j,j+1,\dots,j+s$ $(s\ge 1$) with $k_j=\cdots=k_{j+s}\eqdef k$, the elements $[\vertwedge_k^m,(a_j)^{(k)},1_{m+d-1}]_H,\dots,[\vertwedge_k^m,(a_{j+s})^{(k)},1_{m+d-1}]_H$ have a unique least upper bound, namely $[\vertwedge_k^m,\vv{\id},1_{m+d-1}][\vertwedge_k^{m+d-1},\vv{\id},1_{m+2(d-1)}]_H$. This is true because $[\vertwedge_k^{m+d-1},\vv{\id},1_{m+2(d-1)}]_H$ is the least upper bound of the $[1_{m+d-1},(a_{j+i})^{(k)},1_{m+d-1}]_H$ ($0\le i\le s$) by Lemma~\ref{lem:perm_lubs}, and then Lemma~\ref{lem:decomp_up_bd} says $[\vertwedge_k^m,\vv{\id},1_{m+d-1}][\vertwedge_k^{m+d-1},\vv{\id},1_{m+2(d-1)}]_H$ is consequently the least upper bound of the $[\vertwedge_k^m,(a_{j+i})^{(k)},1_{m+d-1}]_H$ ($0\le i\le s$).

Replacing any such subcollection of our length-$1$ splittings by their least upper bound we get a new collection of elements of $P_H^m$ with the same least upper bound as the original collection. Now our elements are of the form $[U_j^{(k_j)}]_H$ ($1\le j\le t$ for some $t$), satisfying $k_1<\cdots<k_t$, with each $U_j$ either of the form $[\vertwedge_1^1,(a_j)^{(1)},1_d]$ for some $a_j\in A$ or else of the form $[\vertwedge_1^1,\vv{\id},1_d][\vertwedge_1^d,\vv{\id},1_{2d-1}]_H$. Here $U_j^{(k_j)}$ means $1_1\oplus\cdots\oplus 1_1\oplus U_j \oplus 1_1\oplus\cdots\oplus 1_1$ with $U_j$ in the $k_j$th spot. In particular Lemma~\ref{lem:decomp_up_bd} says that every common upper bound of the $[U_j^{(k_j)}]_H$ is an upper bound of $[U_1^{(k_1)}\oplus\cdots\oplus U_t^{(k_t)}]_H$, where we do not write trivial direct summands. This is certainly itself a common upper bound, so it is the least upper bound.
\end{proof}

Note that the proof of Proposition~\ref{prop:lubs} is constructive, and the least upper bound obtained in the end has a rather specific form. Namely, it amounts to first applying length-$1$ single-head splittings to some of the feet of $x$, creating new feet in clumps of $d$ many, potentially with artifacts on the leftmost new foot of a given clump, and then possibly applying $[\vertwedge_1^1,\vv{\id},1_d]$ to the leftmost foot of some clumps. While every least upper bound of some collection of length-$1$ splittings is of this form, not every splitting of this form is necessarily a least upper bound of some such collection.

Now we can finally start defining the $H$-Stein--Farley complex. For $[x]_H\le [y]_H$ in $P_H^m$, define the \emph{closed interval} and \emph{open interval}
$$[[x]_H,[y]_H]\defeq \{[z]_H \mid [x]_H\le [z]_H\le [y]_H\} \text{ and } ([x]_H,[y]_H)\defeq \{[z]_H \mid [x]_H< [z]_H< [y]_H\}\text{,}$$
with half-open intervals defined analogously. The geometric realizations $|[[x]_H,[y]_H]|$ and $|([x]_H,[y]_H)|$ are natural subcomplexes of $|P_H^m|$. The \emph{length} of a closed interval $|[[x]_H,[y]_H]|$ is the length of the splitting $x^{-1}y$.

\begin{definition}[Elementary core]
 Let $[x]_H<[y]_H$ in $P_H^m$. Define the \emph{elementary core} of the closed interval $I=[[x]_H,[y]_H]$ to be
$$\core(I)\defeq lub(I\cap\{[x\Omega]_H \mid \Omega\text{ is a length-$1$ splitting of }x\})\text{.}$$
Also declare that $\core(\{[x]_H\})=[x]_H$ for any $[x]_H$. Proposition~\ref{prop:lubs} ensures that every closed interval in $P_H^m$ has a unique elementary core. Since $[y]_H$ is an upper bound of everything in $I$, the elementary core of $I$ lies in $I$.
\end{definition}

\begin{definition}[Elementary interval/splitting, $\preceq$]
 For $I=[[x]_H,[y]_H]$, if $\core(I)=[y]_H$ call $I$ an \emph{elementary interval}. For any $[x]_H\le [y]_H$ in $P_H^m$, if $[x]_H$ and $[y]_H$ lie in a common elementary interval then write $[x]_H\preceq [y]_H$. If $[x]_H\preceq [y]_H$ and $[x]_H\ne [y]_H$ write $[x]_H\prec [y]_H$. Whenever $[x]_H\preceq [x\Lambda]_H$ for $\Lambda$ a splitting, we call $\Lambda$ an \emph{elementary splitting (of $x$)}.
\end{definition}

Since we know what the least upper bounds in the proof of Proposition~\ref{prop:lubs} look like, we also now know what elementary splittings of $x$ look like. An elementary splitting of $x$ consists of first applying length-$1$ single-head splits to some of the feet of $x$ and then possibly applying the single split $[\vertwedge_1^1,\vv{\id},1_d]$ to the leftmost new foot in some of the clumps of $d$ new feet.

\begin{definition}[$H$-Stein--Farley complex]
 Let $\Stein{S_*\wr G}_H$ be the subcomplex of $|P_H^1|$ consisting only of those chains $[x_0]_H<\cdots<[x_k]_H$ such that $[x_0]_H \prec [x_k]_H$, i.e., such that the chain lies in an elementary interval. Call such a chain/simplex \emph{elementary}, so $\Stein{S_*\wr G}_H$ is the subcomplex of elementary simplices. Note that a subchain of an elementary chain is elementary, so this really is a subcomplex.
\end{definition}

When $H=G$ and $d=2$, $\Stein{S_*\wr G}_G$ is the complex called the \emph{Stein--Farley complex} in \cite{witzel18}, and has the structure of a cube complex. (The name comes from work of Stein \cite{stein92} and Farley \cite{farley03} first establishing these complexes for the classical Thompson groups.) In \cite{belk16} Belk and Matucci use what we would call the $H$-Stein--Farley complex for the R\"over group, using a Klein-$4$ subgroup as $H$, to prove that the R\"over group is $\F_\infty$. Belk and Matucci do not phrase elementary cores as being least upper bounds of length-$1$ subsplittings, but the definitions do coincide in the R\"over group case, as the work in the proof of \cite[Lemma~4.7]{belk16} amounts to proving that the ``double splitting'' is the least upper bound of the two single splittings. As a remark, if $(H,A)$ is not orderly, for example when $G$ is the Gupta--Sidki group, it is difficult to say what an analog of the $H$-Stein--Farley complex should look like, since we would not necessarily have well defined elementary cores.

Before we prove that $|P_H^1|\simeq \Stein{S_*\wr G}_H$ (Proposition~\ref{prop:poset_stein}) we need to collect some observations and notation.

For $[x]_H\in P_H^1$ define $(P_H^1)^{[x]_H \le} \defeq \{[y]_H\in P_H^1\mid [x]_H\le [y]_H\}$. Define a function $\core_{[x]_H} \colon (P_H^1)^{[x]_H \le} \to (P_H^1)^{[x]_H \le}$ by sending $[y]_H$ to $\core([[x]_H,[y]_H])$.

\begin{observation}\label{obs:poset_map}
 For $[x]_H \in P_H^1$, the function $\core_{[x]_H}\colon (P_H^1)^{[x]_H \le}\to (P_H^1)^{[x]_H \le}$ is a poset map.
\end{observation}

\begin{proof}
 Let $[y]_H,[z]_H\in (P_H^1)^{[x]_H \le}$ with $[y]_H\le [z]_H$. Since $[[x]_H,[y]_H]\subseteq [[x]_H,[z]_H]$ we have $\core([[x]_H,[y]_H])\le\core([[x]_H,[z]_H])$, and hence $\core_{[x]_H}([y]_H)\le \core_{[x]_H}([z]_H)$ as desired.
\end{proof}

\begin{proposition}\label{prop:poset_stein}
 If $H$ is $A$-coarsely self-similar for some finite $A\subseteq V_d\cap\Aut(\tree_d)$ and $(H,A)$ is orderly, then $|P_H^1|\simeq \Stein{S_*\wr G}_H$.
\end{proposition}

\begin{proof}
 This general proof technique was first established in \cite{brown92} for the Stein--Farley complex for $F$, and used subsequently in, e.g., \cite{fluch13,bux16,witzel18,belk16}. In particular see \cite[Lemma~4.10 and Proposition~4.11]{belk16}. We will build up from $\Stein{S_*\wr G}_H$ to all of $|P_H^1|$ by attaching all the intervals $|[[x]_H,[y]_H]|$ that do not lie in $\Stein{S_*\wr G}_H$, in order of increasing length, and proving the relative link is contractible whenever such an interval is glued in. If the interval $|[[x]_H,[y]_H]|$ fails to lie in $\Stein{S_*\wr G}_H$ then $[x]_H\not\prec [y]_H$. Since we attach such intervals in increasing order of length, and since every proper subinterval of an interval has strictly shorter length, the relative link when we attach $|[[x]_H,[y]_H]|$ consists of all its simplices corresponding to chains that do not use both $[x]_H$ and $[y]_H$. In other words, the relative link is $|[[x]_H,[y]_H)|\cup|([x]_H,[y]_H]|$. This is the suspension of $|([x]_H,[y]_H)|$, so the relative link is contractible if and only if $|([x]_H,[y]_H)|$ is contractible, which we claim is true whenever $[x]_H\not\preceq [y]_H$.

 Let $[z]_H\in ([x]_H,[y]_H)$. We have $[x]_H<\core_{[x]_H}([z]_H)$ since every non-trivial $[[x]_H,[z]_H]$ contains at least one $[x\Omega]_H$ for $\Omega$ a length-$1$ splitting of $[x]_H$. Also, $\core_{[x]_H}([z]_H)<[y]_H$ since $\core_{[x]_H}([z]_H)\le [z]_H<[y]_H$. In particular $\core_{[x]_H}$ restricts to a function from $([x]_H,[y]_H)$ to itself, which is a poset map by Observation~\ref{obs:poset_map} and which satisfies $\core_{[x]_H}([z]_H)\le [z]_H$ for all $[z]_H\in ([x]_H,[y]_H)$. Moreover, since $[x]_H\not\preceq [y]_H$ we know $\core_{[x]_H}([y]_H)<[y]_H$, so $\core_{[x]_H}([y]_H)\in ([x]_H,[y]_H)$, and every $[z]_H\in ([x]_H,[y]_H)$ satisfies $\core_{[x]_H}([z]_H)\le\core_{[x]_H}([y]_H)$, so \cite[Section~1.5]{quillen78} says that $|([x]_H,[y]_H)|$ is contractible, with $\core_{[x]_H}([y]_H)$ as a cone point.

Having shown that under our procedure every interval $|[[x]_H,[y]_H]|$ in $P_H^1$ not contained in $\Stein{S_*\wr G}_H$ is attached along a contractible relative link, we conclude that $|P_H^1|\simeq \Stein{S_*\wr G}_H$.
\end{proof}

\begin{corollary}\label{cor:stein_cible}
 If $H$ is $A$-coarsely self-similar for some finite $A\subseteq V_d\cap\Aut(\tree_d)$, $(H,A)$ is orderly and $H$ is nuclear in $G$, then $\Stein{S_*\wr G}_H$ is contractible.
\end{corollary}

\begin{proof}
 Since $H$ is nuclear in $G$, $|P_H^1|$ is contractible by Lemma~\ref{lem:poset_cible}. Hence $\Stein{S_*\wr G}_H$ is contractible by Proposition~\ref{prop:poset_stein}.
\end{proof}

\subsection{Polysimplicial structure}\label{sec:poly}

The complex $\Stein{S_*\wr G}_H$ is still not quite nice enough to get higher connectivity of descending links. In this subsection, following \cite{belk16} we will glom certain simplices in $\Stein{S_*\wr G}_H$ together to reveal a coarser, polysimplicial cell structure.

\begin{definition}[Polysimplicial complex]
 A \emph{polysimplex} is a direct product of some finite collection of simplices. A \emph{polysimplicial complex} is an affine cell complex where every cell is a polysimplex, and the intersection of any two cells is a (possibly empty) face of each.
\end{definition}

In dimensions $0$ and $1$ a polysimplex is just a simplex, and in dimension $2$ it is either a triangle or a square.

Recall our convention of leaving out trivial direct summands, and using superscripts like $~^{(i)}$ to indicate where a given summand belongs. In particular if $\Lambda$ is a $1$-head splitting then $\Lambda^{(i)}$ denotes $1_1\oplus\cdots\oplus 1_1\oplus \Lambda \oplus 1_1\oplus\cdots\oplus 1_1$ with $\Lambda$ in the $i$th spot and some number of $1_1$ summands.

\begin{definition}[Simple splitting, disjoint]
 Call a splitting \emph{simple} if it is $\sim_H$-equivalent to one of the form $\Lambda^{(i)}$, for some $i$ and some $1$-head splitting $\Lambda$. Two non-trivial simple splittings $\Lambda_1^{(i)}$ and $\Lambda_2^{(j)}$ with the same number of heads will be called \emph{disjoint} if $i\ne j$. This is well defined since $\Lambda_1^{(i)}\sim_H \Lambda_2^{(j)}$ implies $i=j$.
\end{definition}

See Figure~\ref{fig:simple} for an example of a simple splitting.

\begin{figure}[htb]
 \centering
 \begin{tikzpicture}[line width=0.8pt]
 \draw[->] (-2,0) -- (-2,1);   \draw[->] (-1,0) -- (-1,1);   \draw[->] (0,0) -- (0,1);   \draw[->] (1,0) -- (1,1);   \draw[->] (2,0) -- (2,1);   \draw[->] (3,0) -- (3,1);
	\node at (0,-.25) {$h$};
	\draw (-.5,-1) -- (0,-.5) -- (.5,-1)   (0,-1.5) -- (.5,-1) -- (1,-1.5)   (-.5,-2) -- (0,-1.5) -- (.5,-2);
 \end{tikzpicture}
 \caption{The splitting $([1_1,h,1_1][T,\id,1_4])^{(3)}$ with $6$ heads, where $T$ is the tree consisting of the part of the picture below the ``$h$''. (Here $d=2$.)}
 \label{fig:simple}
\end{figure}
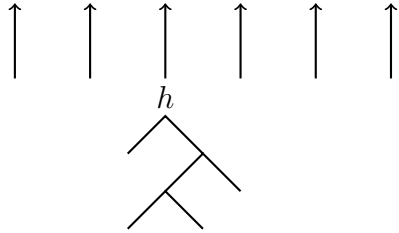

In particular every simple elementary splitting of $x$ consists of first applying a length-$1$ single-head splitting to one of the feet of $x$ and then possibly applying $[\vertwedge_1^1,\vv{\id},1_d]$ to the leftmost new foot. Note that the simple splitting in Figure~\ref{fig:simple} is therefore not elementary, since it features a caret below the second level.

Inspired by \cite[Section~5]{belk16}, we now define a polysimplicial complex $\Stein{S_*\wr G}_H^{poly}$, which will turn out to have $\Stein{S_*\wr G}_H$ as a simplicial subdivision. Let $[x]_H\in P_H^1$, say with $m$ feet. For each $1\le i\le m$ let $S_i$ equal one of the following \emph{elementary sets}:
\begin{align*}
&\{1_1\}\text{, } \\
&\{1_1,[1_1,h,1_1][\vertwedge_1^1,\vv{\id},1_d]\}\text{,} \\
&\{1_1,[1_1,h,1_1][\vertwedge_1^1,\vv{\id},1_d],[\vertwedge_1^1,\vv{\id},1_d][\vertwedge_1^d,\vv{\id},1_{2d-1}]\}\text{, } \\
&\{1_1,[\vertwedge_1^1,\vv{\id},1_d][\vertwedge_1^d,\vv{\id},1_{2d-1}]\}
\end{align*}
for $h$ some element of $H$. If $A=\{\id\}$, i.e., if $H$ is actually self-similar, then we insist the $S_i$ each only be of the first two forms, and that only these two are called \emph{elementary sets}. This ensures that, in any case, all the splittings in an elementary set are elementary. Also note that for any choice of elementary sets $S_i$, the set of $\sim_H$-equivalence classes of elements of each $S_i$ is totally ordered under $\le$. Denote by $\langle x\mid S_1,\dots,S_m\rangle$ the full subcomplex of $\Stein{S_*\wr G}_H$ spanned by all vertices of the form $[x(s_1\oplus\cdots\oplus s_m)]_H$ for $s_i\in S_i$.

For any $x$ and $S_1,\dots,S_m$, $\langle x\mid S_1,\dots,S_m\rangle$ is a simplicial subdivision of the polysimplex obtained by taking the product of the simplices $\langle x\mid \{1_1\},\dots,\{1_1\},S_i,\{1_1\},\dots,\{1_1\}\rangle$. That these are simplices is because $S_i/\sim_H$ is totally ordered. Also note that these simplices have dimension $|S_i|-1$, so more precisely $\langle x\mid S_1,\dots,S_m\rangle$ is a simplicial subdivision of a polysimplex of dimension $(|S_1|-1)\cdots(|S_m|-1)$. We will denote the unsubdivided polysimplex also by $\langle x\mid S_1,\dots,S_m\rangle$, and the abuse of notation will not be a problem. Note that $[x]_H$ is the unique minimal vertex of $\langle x\mid S_1,\dots,S_m\rangle$. Also, if $y=x(s_1\oplus\cdots\oplus s_m)$ for $s_i\in S_i$ then the subcomplex of $\langle x\mid S_1,\dots,S_m\rangle$ spanned by the upper bounds of $[y]_H$ is a face of $\langle x\mid S_1,\dots,S_m\rangle$.

Define $\Stein{S_*\wr G}_H^{poly}$ to be the polysimplicial complex whose cells are the $\langle x\mid S_1,\dots,S_m\rangle$.

\begin{lemma}\label{lem:intersect_polysimplices}
 The intersection of any two polysimplices in $\Stein{S_*\wr G}_H^{poly}$ is a (possibly empty) face of each, so $\Stein{S_*\wr G}_H^{poly}$ really is a polysimplicial complex.
\end{lemma}

\begin{proof}
 We will prove this by closely mimicking the proof of Lemma~5.3 in \cite{belk16} for the case of the R\"over group. Say the polysimplices are $P=\langle x\mid S_1,\dots,S_m\rangle$ and $Q=\langle y\mid T_1,\dots,T_n\rangle$, and assume $P\cap Q\ne \emptyset$. Define a binary operation $\wedge_P\colon P^{(0)}\times P^{(0)}\to P^{(0)}$ via
$$[x(s_1\oplus\cdots\oplus s_m)]_H \wedge_P [x(s_1'\oplus\cdots\oplus s_m')]_H \defeq [x(\min(s_1,s_1')\oplus\cdots\oplus\min(s_m,s_m'))]_H \text{.}$$
Since each elementary set $S_i$ is totally ordered, these minima exist, and since $S_i\to S_i/\sim_H$ is bijective $\wedge_P$ is well defined. Define $\wedge_Q$ analogously. Viewing $P$ and $Q$ as subcomplexes of $\Stein{S_*\wr G}_H$ their intersection is a subcomplex, and the vertex sets of $\Stein{S_*\wr G}_H$ and $\Stein{S_*\wr G}_H^{poly}$ are the same, so $P\cap Q\ne\emptyset$ implies that $P\cap Q$ contains a vertex. We now claim that $\wedge_P$ and $\wedge_Q$ coincide on vertices of $P\cap Q$.

Let $v$ and $v'$ be vertices in $P\cap Q$. Up to replacing $P$ and $Q$ by faces, we can assume that $v\wedge_P v'=[x]_H$ and $v\wedge_Q v'=[y]_H$. Choose $s_i,s_i'\in S_i$ ($1\le i\le m$) and $t_i,t_i'\in T_i$ ($1\le i\le n$) such that
\begin{align*}
 v&=[x(s_1\oplus\cdots\oplus s_m)]_H=[y(t_1\oplus\cdots\oplus t_n)]_H \text{ and}\\
 v'&=[x(s_1'\oplus\cdots\oplus s_m')]_H=[y(t_1'\oplus\cdots\oplus t_n')]_H\text{.}
\end{align*}
Solving for $x^{-1}y$ in the first equation, we get $x^{-1}y=[s_1\oplus\cdots\oplus s_m][1_r,\sigma\vv{h},1_r][(t_1\oplus\cdots\oplus t_n)^{-1}]$ for some $r$ and some $\sigma\vv{h}\in S_r\wr H$. Solving for $x^{-1}y$ in the second equation, we get $x^{-1}y=[s_1'\oplus\cdots\oplus s_m'][1_p,\sigma'\vv{h}',1_p][(t_1'\oplus\cdots\oplus t_n')^{-1}]$ for some $p$ and some $\sigma'\vv{h}'\in S_p\wr H$. In summary,
$$x^{-1}y=[s_1\oplus\cdots\oplus s_m][1_r,\sigma\vv{h},1_r][(t_1\oplus\cdots\oplus t_n)^{-1}]=[s_1'\oplus\cdots\oplus s_m'][1_p,\sigma'\vv{h}',1_p][(t_1'\oplus\cdots\oplus t_n')^{-1}]\text{.}$$
Since $v\wedge_P v'=[x]_H$ we know that for each $1\le i\le m$ either $s_i=1_1$ or $s_i'=1_1$. Similarly for each $1\le i\le n$ either $t_i=1_1$ or $t_i'=1_1$. Now view $x^{-1}y$ as a homeomorphism from $C_1\sqcup\cdots\sqcup C_n$ to $D_1\sqcup\cdots\sqcup D_m$, with $C_i$ and $D_i$ equal to $\partial\tree_d$ for all $i$, as per Remark~\ref{rmk:gpoid_cantor}. Whenever $t_i=1_1$ or $t_i'=1_1$, the image of $C_i$ lands in some $D_j$. In fact this happens for every $i$. Similarly for each $i$ either $s_i$ or $s_i'$ is $1_1$, which means that the image of any $D_j$ under $y^{-1}x$ lands in some $C_i$. This tells us that $m=n$, and that $t_i=1_1$ implies $s_{\sigma(i)}=1_1$ and similarly $t_i'=1_1$ implies $s_{\sigma'(i)}'=1_1$. Since at least one of these holds for each $i$, we conclude that, for each $i$, $x^{-1}y$ acts on $C_i$ by mapping it homeomorphically to $D_{\sigma(i)}$ via a homeomorphism in $H$. Hence $[x^{-1}y]_H=[1_m]_H$ and $[x]_H=[y]_H$, which finishes the proof of the claim that $\wedge_P$ and $\wedge_Q$ coincide on $P\cap Q$.

Let $w_1,\dots,w_k$ be the vertices of $P\cap Q$, and set $w_P\defeq w_1\wedge_P\cdots\wedge_P w_k$ and $w_Q\defeq w_1\wedge_Q\cdots\wedge_Q w_k$. Since $\wedge_P$ and $\wedge_Q$ coincide on $P\cap Q$, $w_P=w_Q$, so call it $w$. Note that $w\le w_j$ for all $1\le j\le k$, and $w$ is itself a vertex of $P\cap Q$. Hence, up to replacing $P$ by a face (the face spanned by the upper bounds of $w$) we can assume $w=[x]_H$, and up to replacing $Q$ by a face we can assume $w=[y]_H$. After this simplification $P\cap Q$ is unchanged but now we have $[x]_H=[y]_H$, and up to possibly adjusting $T_1,\dots,T_m$ we can assume $x=y$, which implies that $P\cap Q=\langle x\mid S_1\cap T_1,\dots,S_m\cap T_m\rangle$. Since any intersection of two elementary sets is again an elementary set, we conclude that $P\cap Q$ is a face of $P$ and $Q$.
\end{proof}

\begin{observation}\label{obs:subdiv}
 $\Stein{S_*\wr G}_H^{poly}$ has $\Stein{S_*\wr G}_H$ as a simplicial subdivision.
\end{observation}

\begin{proof}
 The polysimplices were constructed by unsubdividing certain subcomplexes of $\Stein{S_*\wr G}_H$, so the only thing to show is that every simplex of $\Stein{S_*\wr G}_H$ lies in some polysimplex. But this is immediate since up to $\sim_H$ every elementary splitting is a direct sum of simple elementary splittings, by Corollary~\ref{cor:decomp_split}.
\end{proof}

In particular if $\Stein{S_*\wr G}_H$ is contractible then $\Stein{S_*\wr G}_H^{poly}$ is too.

\subsection{Descending links}\label{sec:dlk}

In this subsection we inspect higher connectivity of descending links in $\Stein{S_*\wr G}_H^{poly}$, and prove our main result that certain R\"over--Nekrashevych groups are of type $\F_\infty$. First we prove that $\feet$ is affine on polysimplices, and hence $\feet$ is a Morse function on $\Stein{S_*\wr G}_H^{poly}$.

\begin{observation}\label{obs:feet_morse}
 The function $\feet$ is affine on polysimplices in $\Stein{S_*\wr G}_H^{poly}$.
\end{observation}

\begin{proof}
 Let $P=\langle x\mid S_1,\dots,S_m\rangle$ be a polysimplex in $\Stein{S_*\wr G}_H^{poly}$. Let $\Sigma_i=\langle x\mid \{1_1\},\dots,S_i,\{1_1\},\dots,\{1\}\rangle$, so $P=\Sigma_1\times\cdots\times \Sigma_m$. Also write $P'$ for the simplicial subdivision of $P$ in $\Stein{S_*\wr G}_H$. Since $\feet$ is invariant under the left action of $\Groupoid{S_*\wr G}$ on $\Stein{S_*\wr G}_H^{poly}$, without loss of generality $x=1_m$. The vertices of $P$ are therefore $\sim_H$-classes of splittings. If $\lambda$ is the function taking such a vertex $v$ to its length as a splitting, then $\feet(v)=m+\lambda(v)(d-1)$. Therefore to see that the extension of $\feet$ from $P^{(0)}$ to $P'$ is affine on $P$, it suffices to show that the extension of $\lambda$ from $P^{(0)}$ to $P'$ is affine on $P$. For $v=(v_1,\dots,v_m)\in P^{(0)}$ with $v_i\in \Sigma_i^{(0)}$ we know $\lambda(v)=\lambda(v_1)+\cdots+\lambda(v_m)$, and $\lambda$ extends affinely to each $\Sigma_i$ since they are simplices, so indeed $\lambda$ (and hence $\feet$) is affine on $P$.
\end{proof}

Fix a vertex $[x]_H$. The notation $\dlk [x]_H$ will always mean descending link in $\Stein{S_*\wr G}_H^{poly}$ (as opposed to $\Stein{S_*\wr G}_H$ or $|P_H^1|$). The link of a vertex in a polysimplicial complex is the union of its links in the polysimplices it lies in, so to understand links in a polysimplicial complex one needs to understand links in polysimplices. The link of a vertex $(v_1,\dots,v_r)$ in a polysimplex $\Sigma_1\times\cdots\times\Sigma_r$ (with each $\Sigma_i$ a simplex) is the join of the links $\lk_{\Sigma_i} v_i$, and these are all easily understood.

Many of the proof ideas in what follows come from the analogous ones for the R\"over group case in \cite[Section~6]{belk16}.

\begin{definition}[Merging]
 An element $\Upsilon\in\Groupoid{S_*\wr H}$ is called a \emph{merging} if $\Upsilon^{-1}$ is a splitting. The \emph{length} of $\Upsilon$ is the length of $\Upsilon^{-1}$. If $\Upsilon^{-1}$ is a simple splitting we call $\Upsilon$ a \emph{simple merging}, and if $\Upsilon^{-1}$ is an elementary splitting we call $\Upsilon$ an \emph{elementary merging}. If $\Upsilon$ is a merging with $\feet([x]_H)$ heads, so the product $x\Upsilon$ makes sense, we call $\Upsilon$ a \emph{merging of $x$}. We also call $[x\Upsilon]_H$ a \emph{merging of $[x]_H$}.
\end{definition}

For $[y]_H<[x]_H$, the vertices $[x]_H$ and $[y]_H$ span an edge in $\Stein{S_*\wr G}_H$ if and only if $[y]_H\prec [x]_H$, that is if and only if $y^{-1}x$ is a non-trivial elementary splitting, but they span an edge in $\Stein{S_*\wr G}_H^{poly}$ if and only if $y^{-1}x$ is a non-trivial simple elementary splitting, i.e., $x^{-1}y$ is a simple elementary merging. Hence the vertices of $\dlk [x]_H$ are those of the form $[x\Upsilon]_H$ for $\Upsilon$ a non-trivial simple elementary merging. Every simple splitting is $\sim_H$-equivalent to a splitting of the form $\Lambda^{(i)}$ for $\Lambda$ a $1$-head splitting. Hence every simple merging is of the form $[1_m,\sigma\vv{h},1_m]\Upsilon^{(i)}$ for some $\sigma\vv{h}\in S_m\wr H$, some $1$-foot merging $\Upsilon$ and some $1\le i\le m-r(d-1)$, where $r$ is the length of $\Upsilon$. For notational ease we will denote this particular simple merging by $\sigma\vv{h}\Upsilon^{(i)}$.

Having understood the vertices of $\dlk [x]_H$, we turn our attention to the higher dimensional simplices. First we discuss a sufficient condition for a collection of vertices to span a simplex in $\dlk [x]_H$, namely if they have pairwise disjoint supports.

\begin{definition}[Support]
 For a non-trivial simple merging $\sigma\vv{h}\Upsilon^{(i)}$ with length $r$, define its \emph{support} to be the set
$$\{\sigma(i')\mid i\le i'\le i+r(d-1)\}\text{.}$$
\end{definition}

For some intuition, if $\Upsilon^{(i)}$ is a length-$r$ simple merging of $x$ then it ``uses'' the feet of $x$ numbered $i$ through $i+r(d-1)$. Hence $\sigma\vv{h}\Upsilon^{(i)}$ ``uses'' the (possibly non-consecutive) feet of $x$ numbered $\sigma(i),\dots,\sigma(i+r(d-1))$. These indices comprise what we are calling the support of the merging. The point, which the next lemma makes rigorous, is that if simple mergings have disjoint supports then they can be done simultaneously, at least up to possibly adjusting them with $\sim_H$.

\begin{lemma}\label{lem:disjoint_merge_span}
 If $\sigma_1\vv{h}_1\Upsilon_1^{(i_1)},\dots,\sigma_p\vv{h}_p\Upsilon_p^{(i_p)}$ are non-trivial simple elementary mergings of $x$ with pairwise disjoint supports then the vertices $[x\sigma_j\vv{h}_j\Upsilon_j^{(i_j)}]_H$ of $\dlk [x]_H$ span a simplex in $\dlk [x]_H$.
\end{lemma}

\begin{proof}
 Thanks to the left action of $\Groupoid{S_*\wr G}$ on $\Stein{S_*\wr G}_H^{poly}$ we can assume $x=1_m$. It suffices to show that the $p+1$ vertices $[1_m]_H$, $[\sigma_1\vv{h}_1\Upsilon_1^{(i_1)}]_H,\dots,[\sigma_p\vv{h}_p\Upsilon_p^{(i_p)}]_H$ of $\Stein{S_*\wr G}_H^{poly}$ lie in a common polysimplex. Since everything is happening up to $\sim_H$, without loss of generality $\vv{h}_j=\vv{\id}$ for all $1\le j\le p$ (the entries not ``trapped'' by $\Upsilon_j$ can vanish in $\sim_H$ and the entries that are trapped can be absorbed into $\Upsilon_j$).

 Let $r_j$ be the length of $\Upsilon_j$. Since the mergings have pairwise disjoint supports we know that $\sum_{j=1}^p|\{\sigma_j(i')\mid i_j\le i'\le i_j+r_j(d-1)\}|\le m$, so $\sum_{j=1}^p(1+r_j(d-1))\le m$. In particular we can choose $p$ pairwise disjoint sets $D_1,\dots,D_p\subseteq \{1,\dots,m\}$ with each $D_j$ consisting of some string of $1+r_j(d-1)$ consecutive numbers. Say $D_j=\{d_j,d_j+1,\dots,d_j+r_j(d-1)\}$. We will assume $d_1=1$ and $d_{j+1}=(d_j+r_j(d-1))+1$ for all $j$, so the $D_j$ are as ``tightly packed to the left'' as possible.

Now define $\tau\in S_m$ to be any permutation satisfying $\tau(k)=\sigma_j(k+i_j-d_j)$ for all $k\in D_j$ and all $1\le j\le p$. A well defined such $\tau$ exists since the $D_j$ are pairwise disjoint, and $\tau$ can be chosen to be injective (hence bijective) since the $\{\sigma_j(i')\mid i_j\le i'\le i_j+r_j(d-1)\}$ are pairwise disjoint. Let
$$\Upsilon\defeq \Upsilon_1^{(1)}\oplus\Upsilon^{(2)}\oplus\cdots\oplus\Upsilon_p^{(p)}\oplus 1_r$$
where $r=m-\sum_{j=1}^p|D_j|$, so $\Upsilon$ has $m$ heads and $p+r$ feet. In particular it makes sense to consider $1_m\tau\Upsilon$. For each $1\le j\le p$ let $S_j\defeq \{1_1,\Lambda_j\}$ be the elementary set satisfying $[\Lambda_j]_H=[\Upsilon_j^{-1}]_H$ (this is possible since $\Upsilon_j$ is elementary). For all $p<j\le p+r$ let $S_j\defeq\{1_1\}$. We want to show that the polysimplex $P\defeq \langle \tau\Upsilon \mid S_1,\dots,S_{p+r}\rangle$ contains all the vertices $[1_m]_H$, $[\sigma_1\Upsilon_1^{(i_1)}]_H,\dots,[\sigma_p\Upsilon_p^{(i_p)}]_H$.

Since $[1_m]_H=[\tau]_H=[\tau\Upsilon(\Lambda_1^{(1)}\oplus\cdots\oplus\Lambda_p^{(p)}\oplus 1_r)]_H$ we know $[1_m]_H\in P$. Now we claim that
$$[\tau\Upsilon(1_1\oplus\Lambda_2^{(2)}\oplus\cdots\oplus\Lambda_p^{(p)}\oplus 1_r)]_H=[\sigma_1\Upsilon_1^{(i_1)}]_H\text{,}$$
which will imply that $[\sigma_1\Upsilon_1^{(i_1)}]_H\in P$. Since the left hand side equals $[\tau\Upsilon_1^{(1)}]_H$ it suffices to show that $[\Lambda_1^{(1)}\tau^{-1}\sigma_1\Upsilon_1^{(i_1)}]_H=[1_{m-r_1(d-1)}]_H$. Write $\Lambda_1$ as $[T,\upsilon\vv{g},1_{1+r_1(d-1)}]$ for some tree $T$ and $\upsilon\vv{g}\in S_{1+r_1(d-1)}\wr G$. Since $\tau^{-1}\sigma_1(i')=i'-i_1+1$ for all $i_1\le i'\le i_1+r_1(d-1)$ we know $[\Lambda_1^{(1)}\tau^{-1}\sigma_1\Upsilon_1^{(i_1)}]_H = [T^{(1)},\tau^{-1}\sigma_1,T^{(i_1)}]_H$, i.e., the $\upsilon\vv{g}$ cancels out. Moreover this tells us that $\tau^{-1}\sigma_1$ lies in the image of the composition $\clone_{T^{(i_1)}}$ of cloning maps given by the $d$-carets of $T^{(i_1)}$, and hence this reduces to $[1_{m-r_1(d-1)},(\tau^{-1}\sigma_1)\clone_{T^{(i_1)}}^{-1},1_{m-r_1(d-1)}]_H=[1_{m-r_1(d-1)}]_H$. This finishes the proof of the claim that $[\sigma_1\Upsilon_1^{(i_1)}]_H \in P$. A parallel argument shows that $[\sigma_j\Upsilon_j^{(i_j)}]_H\in P$ for all $1\le j\le p$.
\end{proof}

In general a pair of distinct vertices $[x\Upsilon]_H,[x\Upsilon']_H \in \dlk [x]_H$ span an edge in $\dlk [x]_H$ if and only if the three vertices $[x]_H$, $[x\Upsilon]_H$ and $[x\Upsilon']_H$ lie in a common $2$-dimensional polysimplex, which happens either when $\Upsilon^{-1}\Upsilon'$ or $(\Upsilon')^{-1}\Upsilon$ is a non-trivial simple elementary merging (so the vertices span a triangle), or when $\Upsilon$ and $\Upsilon'$ have disjoint supports (so the vertices lie in a common square, as per the proof of Lemma~\ref{lem:disjoint_merge_span} for $p=2$). In the former case the supports of $\Upsilon$ and $\Upsilon'$ must be properly nested, so we get the following:

\begin{observation}\label{obs:nest_or_disjoint}
 If $[x\Upsilon]_H,[x\Upsilon']_H \in \dlk [x]_H$ span an edge in $\dlk [x]_H$ then the supports of $\Upsilon$ and $\Upsilon'$ are either properly nested or disjoint. If they are properly nested then $\Upsilon^{-1}\Upsilon'$ or $(\Upsilon')^{-1}\Upsilon$ is a non-trivial simple elementary merging. \qed
\end{observation}

Also note that, since simple elementary mergings have length at most $2$, we cannot have $[x\Upsilon]_H,[x\Upsilon']_H,[x\Upsilon'']_H \in \dlk [x]_H$ with supports pairwise properly nested.

Recall that a simplicial complex is a \emph{flag complex} if every collection of vertices that pairwise span edges spans a simplex.

\begin{lemma}
 $\dlk [x]_H$ is a flag complex.
\end{lemma}

\begin{proof}
 Thanks to the left action of $\Groupoid{S_*\wr G}$ on $\Stein{S_*\wr G}_H^{poly}$ we can assume $x=1_m$. Let $\sigma_1\vv{h}_1\Upsilon_1^{(i_1)},\dots,\sigma_q\vv{h}_q\Upsilon_q^{(i_q)}$ be non-trivial simple elementary mergings of $1_m$ such that the vertices $[\sigma_j\vv{h}_j\Upsilon_j^{(i_j)}]_H$ of $\dlk [1_m]_H$ pairwise span edges. We need to show that all these vertices and $[1_m]_H$ lie in a common polysimplex. As in the proof of Lemma~\ref{lem:disjoint_merge_span} we can assume $\vv{h}_j=\vv{\id}$ for all $j$. By Observation~\ref{obs:nest_or_disjoint} the supports of the $\sigma_j\Upsilon_j^{(i_j)}$ are pairwise either properly nested or disjoint. Denote the support of $\sigma_j\Upsilon_j^{(i_j)}$ by $U_j$. Assume without loss of generality that the $\sigma_j\Upsilon_j^{(i_j)}$ are ordered so that there exists $1\le p\le q$ such that $U_1,\dots,U_p$ are pairwise disjoint and for each $p<j\le q$ the set $U_j$ is properly contained in one of the $U_1,\dots,U_p$. Thanks to Lemma~\ref{lem:disjoint_merge_span} there is a polysimplex containing $[1_m]_H$ and $[\sigma_j\Upsilon_j^{(i_j)}]_H$ for all $1\le j\le p$ (but not necessarily for $p<j\le q$). Say the polysimplex is $P=\langle \tau\Upsilon\mid S_1,\dots,S_{p+r}\rangle$, with $\tau$, $\Upsilon$, $r$ and $S_j$ as in the proof of Lemma~\ref{lem:disjoint_merge_span}. In particular for $1\le j\le p$ we have $S_j=\{1_1,\Lambda_j\}$ with $[\Lambda_j]_H=[\Upsilon_j^{-1}]_H$.

 If $p=q$ we are done so suppose $p<q$. Each of $U_{p+1},\dots,U_q$ is properly contained in one of $U_1,\dots,U_p$. Also, for each $1\le j\le p$ at most one of $U_{p+1},\dots,U_q$ can be properly contained in $U_j$, since simple elementary mergings have length at most $2$. Hence $q-p\le p$, and without loss of generality $U_{p+j}\subsetneq U_j$ for each $1\le j\le q-p$. Now for each such $j$, we know $\Big(\sigma_{p+j}\Upsilon_{p+j}^{(i_{p+j})}\Big)^{-1}\sigma_j\Upsilon_j^{(i_j)}$ is a simple elementary merging, hence its inverse is a simple elementary splitting, so we can choose $\Lambda_j'$ such that $[(\Lambda_j')^{(i_j)}]_H=\Big[\Big(\sigma_j\Upsilon_j^{(i_j)}\Big)^{-1}\sigma_{p+j}\Upsilon_{p+j}^{(i_{p+j})}\Big]_H$ and $S_j'\defeq\{1_1,\Lambda_j',\Lambda_j\}$ is an elementary set. For each $q-p<k\le q$ set $S_k'\defeq S_k$. Now let $Q=\langle \tau\Upsilon\mid S_1',\dots,S_{p+r}'\rangle$. Since $P\subseteq Q$ we know $[1_m]_H$ and $[\sigma_j\Upsilon_j^{(i_j)}]_H$ are in $Q$ for all $1\le j\le p$. Now let $p<p+j\le q$. Then
\begin{align*}
 & [\sigma_{p+j}\Upsilon_{p+j}^{(i_{p+j})}]_H \\
 =& \Big[\sigma_j\Upsilon_j^{(i_j)}\Big(\sigma_j\Upsilon_j^{(i_j)}\Big)^{-1}\sigma_{p+j}\Upsilon_{p+j}^{(i_{p+j})}\Big]_H \\
 =& [\sigma_j\Upsilon_j^{(i_j)}(\Lambda_j')^{(i_j)}]_H \\
 =& [\tau\Upsilon(\Lambda_1^{(1)}\oplus\cdots\oplus\Lambda_{j-1}^{(j-1)}\oplus (\Lambda_j')^{(j)}\oplus \Lambda_{j+1}^{(j+1)}\oplus\cdots\oplus \Lambda_p^{(p)}\oplus 1_r)]_H\text{.}
\end{align*}
Hence $[\sigma_{p+j}\Upsilon_{p+j}^{(i_{p+j})}]_H\in Q$ and we are done.
\end{proof}

In order to understand higher connectivity of these descending links, we will use the following result, due to Belk and Forrest \cite[Theorem~4.9]{belkforrest15}, which Belk and Matucci used in \cite{belk16} in the R\"over group case.

\begin{cit}\cite[Theorem~6.2]{belk16}\label{cit:ground}
 Let $\Delta$ be a flag complex and let $c,k\ge 1$. Suppose $\Delta$ admits a $ck$-simplex $\Sigma$ such that for every vertex $v$ of $\Delta$, $v$ shares an edge with all but at most $k$ vertices of $\Sigma$. Then $\Delta$ is $(c-1)$-connected.
\end{cit}

In \cite{belkforrest15} and \cite{belk16} $\Delta$ is assumed to be finite, but since every homotopy sphere in a flag complex lies in a finite flag subcomplex, the result is equally true for infinite $\Delta$.

We will exhibit a simplex $\Sigma$ in $\dlk [x]_H$ that will end up working. Set $m\defeq\feet([x]_H)$, and for each $0\le i\le \lfloor\frac{m}{d}\rfloor-1$ let $\mho_i\defeq [1_m,\vv{\id},\vertwedge_{1+di}^m]$, so the support of $\mho_i$ is $\{di+1,\dots,di+d\}$. Since these are disjoint for different $i$, Lemma~\ref{lem:disjoint_merge_span} says that all the $[x\mho_i]_H$ span a simplex $\Sigma$ in $\dlk [x]_H$. The dimension of $\Sigma$ is $1$ less than its number of vertices, i.e., $\lfloor\frac{m}{d}\rfloor-1$.

See Figure~\ref{fig:ground} for an example of this $\Sigma$.

\begin{figure}[htb]
 \centering
 \begin{tikzpicture}[line width=0.8pt]
 \draw (0,0) -- (.75,-.5) -- (1.5,0)   (.5,0) -- (.75,-.5) -- (1,0)   (2,0) -- (2.75,-.5) -- (3.5,0)   (2.5,0) -- (2.75,-.5) -- (3,0)   (4,0) -- (4.75,-.5) -- (5.5,0)   (4.5,0) -- (4.75,-.5) -- (5,0);
 \end{tikzpicture}
 \caption{For $x=1_{12}$ and $d=4$ (so $r=0$), the supports of $\mho_0$, $\mho_1$ and $\mho_2$ are $\{1,2,3,4\}$, $\{5,6,7,8\}$ and $\{9,10,11,12\}$. This picture is of $[\mho_0\oplus\mho_1\oplus\mho_2]$, which describes (the barycenter of) the simplex $\Sigma$.}
 \label{fig:ground}
\end{figure}
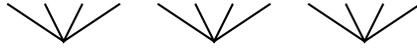

To see that this $\Sigma$ works, we need to understand when a vertex of of $\dlk [x]_H$ shares an edge with a vertex of $\Sigma$.

\begin{lemma}\label{lem:compatible}
 Every vertex of $\dlk [x]_H$ shares an edge with all but at most $2(d-1)+1$ vertices of $\Sigma$.
\end{lemma}

\begin{proof}
 Let $[x\sigma\vv{h}\Upsilon^{(j)}]_H$ be an arbitrary vertex of $\dlk [x]_H$, with $\sigma\vv{h}\in S_m\wr H$ (for $m=\feet(x)$) and $\Upsilon$ a non-trivial $1$-foot elementary merging. Say $\Upsilon$ has length $r$, so $r\le 2$. We claim that the support of $\sigma\vv{h}\Upsilon^{(j)}$ is disjoint from the supports of all but at most $2(d-1)+1$ of the $\mho_i$. The support of $\mho_i$ is $\{di+1,\dots,di+d\}$ and the support of $\sigma\vv{h}\Upsilon^{(j)}$ is $\{\sigma(j')\mid j\le j'\le j+r(d-1)\}$. Since the $\{di+1,\dots,di+d\}$ are themselves disjoint for different $i$ values, $\{\sigma(j')\mid j\le j'\le j+r(d-1)\}$ can intersect $\{di+1,\dots,di+d\}$ for at most $|\{\sigma(j')\mid j\le j'\le j+r(d-1)\}|=r(d-1)+1\le 2(d-1)+1$ values of $i$.
\end{proof}

As a remark, this would not have worked if we had not passed from $\Stein{S_*\wr G}_H$ to $\Stein{S_*\wr G}_H^{poly}$. Simple elementary merges have bounded size of support, whereas elementary merges in theory can have arbitrarily large supports (here support means the union of supports of the simple factors). For example if $x$ has $md$ feet then the (non-simple) elementary merging $[1_d,\vv{\id},\vertwedge_1^d]\oplus\cdots\oplus[1_d,\vv{\id},\vertwedge_1^d]$ (with $m$ summands) has support equal to all of $\{1,\dots,m\}$.

Now we can prove our higher connectivity result.

\begin{proposition}[Higher connectivity of descending links]\label{prop:desc_lks_hi_conn}
 If $H$ is $A$-coarsely self-similar for some finite $A\subseteq V_d\cap\Aut(\tree_d)$ and $(H,A)$ is orderly, then for each $n\in\N$ there exists $t\in\R$ such that for all vertices $[x]_H$ in $\Stein{S_*\wr G}_H^{poly}$ with $\feet([x]_H)\ge t$ we have that $\dlk [x]_H$ is $n$-connected.
\end{proposition}

\begin{proof}
 Since $\Sigma$ has dimension $\lfloor\frac{\feet([x]_H)}{d}\rfloor-1$, Citation~\ref{cit:ground} and Lemma~\ref{lem:compatible} say that $\dlk [x]_H$ is
$$\left(\left\lfloor\frac{\left\lfloor\frac{\feet([x]_H)}{d}\right\rfloor-1}{2(d-1)+1}\right\rfloor-1\right)-\text{connected,}$$
so if we take $t$ large enough that
$$\left\lfloor\frac{\left\lfloor\frac{t}{d}\right\rfloor-1}{2(d-1)+1}\right\rfloor\ge n+1$$
then we get that $\feet([x]_H)\ge t$ implies $\dlk [x]_H$ is $n$-connected.
\end{proof}

We can now prove our main theorem about R\"over--Nekrashevych groups of type $\F_\infty$.

\begin{theorem}\label{thrm:good_H}
 Let $G\le \Aut(\tree_d)$ be self-similar and suppose there exists $H\le G$ such that $H$ is nuclear in $G$ (see Definition~\ref{def:nuclear}) and $A$-coarsely self-similar (see Definition~\ref{def:Acss}) for some finite $A\subseteq V_d\cap\Aut(\tree_d)$, and $(H,A)$ is orderly (see Definition~\ref{def:orderly}). If $H$ is of type $\F_\infty$ then so is $V_d(G)$.
\end{theorem}

\begin{proof}
 Consider the action of $V_d(G)$ on $\Stein{S_*\wr G}_H^{poly}$. Since $\Stein{S_*\wr G}_H^{poly}$ has $\Stein{S_*\wr G}_H$ as a simplicial subdivision (Observation~\ref{obs:subdiv}), it is contractible (Corollary~\ref{cor:stein_cible}) so we can try and apply Lemma~\ref{lem:morse} using the Morse function $\feet$ (see Observation~\ref{obs:feet_morse}). Since sublevel sets in $|P_H^1|$ are cocompact (Lemma~\ref{lem:cocpt}), the same is true in $\Stein{S_*\wr G}_H^{poly}$. Since stabilizers in $V_d(G)$ of vertices in $|P_H^1|$ are of type $\F_\infty$ (Lemma~\ref{lem:stabs}), the same is true in $\Stein{S_*\wr G}_H^{poly}$. The proof in Corollary~\ref{cor:stabs} that every stabilizer of a simplex in $|P_H^1|$ has finite index in some vertex stabilizer also establishes that every stabilizer of a polysimplex in $\Stein{S_*\wr G}_H^{poly}$ has finite index in a vertex stabilizer, and so all cell stabilizers are of type $\F_\infty$. Finally Proposition~\ref{prop:desc_lks_hi_conn} establishes the last condition of Lemma~\ref{lem:morse}.
\end{proof}

When $G$ itself is of type $\F_\infty$, all the necessary conditions are met almost trivially:

\begin{corollary}\label{cor:Finfty_inherited}
 Let $G\le \Aut(\tree_d)$ be self-similar and of type $\F_\infty$. Then $V_d(G)$ is of type $\F_\infty$.
\end{corollary}

\begin{proof}
 Since $G$ is self-similar it is $\{\id\}$-coarsely self-similar (and $(G,\{\id\})$ is trivially orderly), and any $G$ is nuclear in itself, so this follows from Theorem~\ref{thrm:good_H} applied with $H=G$.
\end{proof}

\begin{remark}
 In \cite{thumann17}, Thumann develops a unified framework, utilizing operads, for proving that a variety of Thompson-like groups are of type $\F_\infty$. It is an interesting question which $V_d(G)$ fit into this framework. We believe that if $G$ itself is of type $\F_\infty$ then one could likely use Thumann's techniques to recover our result that $V_d(G)$ is too. For the case when $G$ is not of type $\F_\infty$, but some orderly $A$-coarsely self-similar subgroup $H$ that is nuclear in $G$ is of type $\F_\infty$, it is much less clear whether one could apply Thumann's approach. We suspect not, since it seems like Thumann's framework would demand that the stabilizers be built out of $G$ rather than getting to choose a nice $H$. One could probably modify Thumann's machine to handle these groups, by first translating the approach here and in \cite{belk16} into the language of operads. It would be interesting to see if using operads could yield $\F_\infty$ for any R\"over--Nekrashevych groups beyond the ones handled here.
\end{remark}


\section{Examples}\label{sec:ex}

In this section we discuss some examples of R\"over--Nekrashevych groups that can be seen to be of type $\F_\infty$ using our results. As before, we have a fixed self-similar $G\le \Aut(\tree_d)$ and a fixed $H\le G$.

\subsection{$G$ of type $\F_\infty$}\label{sec:G_Finfty}

First we consider the case when $G$ is of type $\F_\infty$, so $V_d(G)$ is also of type $\F_\infty$ thanks to Corollary~\ref{cor:Finfty_inherited}. Since finite groups are of type $\F_\infty$, we have the following example, which is worth remarking on first.

\begin{example}[FSS R\"over--Nekrashevych groups]\label{ex:fss_Finfty}
 Besides the R\"over group, another previously handled example from the literature concerns R\"over--Nekrashevych groups that are FSS groups, as in \cite{farley15,hughes09}. As in Example~\ref{ex:fss_gp}, let $D\le S_d$ and embed $D$ into $\Aut(\tree_d)$ by by viewing the vertices of $\tree_d$ as words in the alphabet $\{1,\dots,d\}$. Let $\iota$ be the embedding and $G$ its image. Then $G$ is self-similar via the wreath recursion $g=(\iota^{-1}(g))(g,\dots,g)$, so it makes sense to consider the R\"over--Nekrashevych group $V_d(G)$. It turns out $V_d(G)$ is the FSS group determined by $G$ \cite[Remark~2.13]{farley15}, and Farley and Hughes proved it is of type $\F_\infty$ \cite[Corollary~6.6]{farley15}. Since such $G$ are finite, they are of type $\F_\infty$, so the Farley--Hughes result is a special case of our results here. Since every finite group embeds in some $S_d$, we can phrase this all as, every finite group produces a R\"over--Nekrashevych group of type $\F_\infty$. As a remark, many $V_d(G)$ are just isomorphic to $V_d$; this happens if and only if the action of $D$ on $\{1,\dots,d\}$ is free \cite[Theorem~1]{bleak17}.
\end{example}

Other examples of self-similar groups of type $\F_\infty$ that have been considered in the literature include free groups \cite{vorobets10,steinberg11}, free products of arbitrarily many copies of $\Z/2\Z$ \cite{savchuk11}, and $(\Z/3\Z)*(\Z/3\Z)$ \cite{bondarenko17}. In all these cases, the groups are not only self-similar but are even so called automaton groups. More generally, we know now from Corollary~\ref{cor:vfree} that every finitely generated virtually free group is self-similar, and of type $\F_\infty$ since finiteness properties are invariant under passing to finite index. Hence our main example here is:

\begin{example}[Virtually free]\label{ex:vfree}
 For any finitely generated virtually free group $G$, we can embed $G$ into $\Aut(\tree_d)$ as a self-similar group for some $d$, and moreover any R\"over--Nekrashevych group $V_d(G)$ for $G$ a finitely generated virtually free group is of type $\F_\infty$.
\end{example}

\subsection{\v Suni\'c groups}\label{sec:sunic}

The Grigorchuk group admits a number of generalizations, including the infinite family of \v Suni\'c groups briefly discussed in Example~\ref{ex:sunic}. Here we discuss these in more detail, and prove that they fit into our framework and hence yield R\"over--Nekrashevych groups of type $\F_\infty$.

\begin{definition}[\v Suni\'c group]
 Let $A$ be the cylic group of order $d$, say $A=\langle a\rangle$, and view $A$ as a subgroup of $\Aut(\tree_d)$ via the embedding $a\mapsto [\vertwedge_1^1,(1\cdots d),\vertwedge_1^1]$. Let $B$ be a finitely generated group and $\rho\colon B\to B$ an automorphism, and suppose there exists a homomorphism $\omega\colon B\to A$ such that no non-trivial $\rho$-orbit lies entirely in $\ker(\omega)$. Map $B$ to $\Aut(\tree_d)$ by declaring the wreath recursion $b=(\omega(b),\id,\dots,\id,\rho(b))$. This defines a faithful action on $\tree_d$ since we have assumed that for any $1_B\ne b\in B$ there exists $n$ such that $\omega(\rho^n(b))\ne 1_A$. Viewing $B$ as a subgroup of $\Aut(\tree_d)$ in this way, define $G_{\omega,\rho}\defeq\langle A,B\rangle$. This is the \emph{\v Suni\'c group} for $\omega$ and $\rho$. (Since $A$ and $B$ are the range and domain of $\omega$ we need not include them in the subscript.)
\end{definition}

Note that $B$ must be abelian and satisfy $B^d=\{1_B\}$, since $[B,B]$ and $B^d$ are characteristic and lie in $\ker(\omega)$. Additionally, the $G_{\omega,\rho}$ act transitively on every level of $\tree_d$ and as such are infinite. These groups were introduced by \v Suni\'c in \cite{sunic07} with the added assumptions that $d$ is prime, and hence $B\cong A^m$ for some $m$. \v Suni\'c also considered specific $\omega$ and $\rho$ arising from certain concrete matrices, and the groups $G_{\omega,\rho}$ for these specific $\omega$ and $\rho$ were called \v Suni\'c groups in \cite{francoeur16}.

\begin{lemma}\label{lem:sunic_works}
 For any \v Suni\'c group $G_{\omega,\rho}$ we have that $B$ is $A$-coarsely self-similar, of type $\F_\infty$, and nuclear in $G_{\omega,\rho}$, and $(B,A)$ is orderly.
\end{lemma}

\begin{proof}
 The wreath recursion $b=(\omega(b),\id,\dots,\id,\rho(b))$ ensures that $B$ is $A$-coarsely self-similar, since $\omega(b)\in A$ and $\rho(b)\in B$. It is also immediate from this wreath recursion that $(B,A)$ is orderly, namely it is of the form in Example~\ref{ex:orderly}. That $B$ is of type $\F_\infty$ follows from the fact it is finitely generated and abelian. It remains to show that $B$ is nuclear in $G_{\omega,\rho}$.

We first claim that $G_{\omega,\rho}$ is contracting with nucleus contained in $A\cup B$. We will induct on length of words in the generating set $(A\cup B)\setminus\{\id\}$. Let $g=a^{n_1}b_1 a^{n_2}b_2\cdots a^{n_k}b_k$ for some $k\ge1$, $n_1\in\{1,\dots,d\}$, $n_2,\dots,n_k\in\{1,\dots,d-1\}$, $\{b_1,\dots,b_{k-1}\}\subseteq B\setminus\{\id\}$ and $b_k\in B$. Note that $g$ has length either $2k-2$, $2k-1$ or $2k$ depending on whether $n_1=d$ and/or $b_k=\id$. As a base case, suppose $k=1$, so $g=a^{n_1}b_1$. Then $g=(1\cdots d)^{n_1}(\omega(b_1),\id,\dots,\id,\rho(b_1))$ and so the states on the first level are already in $A\cup B$. Now suppose $k\ge2$. We have
$$g=a^{n_1+n_2+\cdots+n_k}(a^{-(n_2+n_3+\cdots+n_k)} b_1 a^{n_2+n_3+\cdots+n_k}) \cdots (a^{-n_k}b_{k-1}a^{n_k})b_k\text{.}$$
Each $a^{-(n_{i+1}+\cdots+n_k)} b_i a^{n_{i+1}+\cdots+n_k}$ equals $(\omega(b_i),\id,\dots,\id,\rho(b_i))$ with the entries cyclically permuted, hence is of the form $(g_i^1,\dots,g_i^d)$ for $g_i^j\in A\cup B$. This means
$$g=(1\cdots d)^{n_1+n_2+\cdots n_k} (g_1^1,\dots,g_1^d)(g_2^1,\dots,g_2^d)\cdots(g_k^1,\dots,g_k^d)$$
with $g_i^j\in A\cup B$ for all $i,j$. In particular, any state on the first level is of the form $g_1^j g_2^j\cdots g_k^j$, and this is a word of length at most $k$. Moreover, if $b_k=\id$ then $g_k^j=\id$ so the length is at most $k-1$. Since $k\ge 2$ and $g$ has length at least $2k-2$ we see that in all cases $g_1^j g_2^j\cdots g_k^j$ is strictly shorter than $g$. The claim now follows by induction.

Now let $a^{n_1}b_1 a^{n_2}b_2\cdots a^{n_k}b_k$ be an arbitrary element of $G_{\omega,\rho}$, for $n_i\in\{1,\dots,d\}$ and $b_j\in B$. Since $G_{\omega,\rho}$ is contracting with nucleus contained in $A\cup B$, there exists an $m$ such that the states of $a^{n_1}b_1 a^{n_2}b_2\cdots a^{n_k}b_k$ on the $m$th level are in $A\cup B$. Suppose that the states on the $m$th level that are in $A\setminus B$ are in the coordinates $k_1,\dots,k_r$. Then $[1_1,a^{n_1}b_1 a^{n_2}b_2\cdots a^{n_k}b_k,1_1]=[T, \sigma \vv{b}, T]$ where $T$ is the tree obtained by taking the finite rooted $d$-ary tree whose leaves are all distance $m$ from the root and attaching a $d$-caret to the $k_i$th leaf for each $1\le i\le r$, and where $\sigma \vv{b} \in S_{d^m+r(d-1)}\wr B$. Since $\vv{b}\in B^{d^m+r(d-1)}$, this finishes the proof that $B$ is nuclear in $G$.
\end{proof}

\begin{corollary}
 For any \v Suni\'c group $G_{\omega,\rho}$, the R\"over--Nekrashevych group $V_d(G_{\omega,\rho})$ is of type $\F_\infty$.
\end{corollary}

\begin{proof}
 This is immediate from Lemma~\ref{lem:sunic_works} and Theorem~\ref{thrm:good_H}.
\end{proof}

Examples that arise in this setting include Grigorchuk's group (as addressed in \cite{belk16}) and the Fabrykowski-Gupta group. The \v Suni\'c groups are all examples of bounded automaton groups and as such are amenable \cite{bartholdi10}. In addition, many of these examples are regular branch groups, and thus are not elementary amenable \cite{juschenko17}. This class contains groups with a number of other exotic properties as well, including infinite torsion groups. The torsion examples, and the examples for $d=2$ that are not isomorphic to the infinite dihedral group, have intermediate word growth \cite{sunic07,francoeur16}.

We note that those \v Suni\'c groups $G$ that are regular branch are of type $\F_1$ but not type $\F_2$ \cite{bartholdi03}, so using $H=G$ would not have worked.

\subsection{Questions}

There is evidence that as soon as $G$ is finitely generated and contracting, $V_d(G)$ is of type $\F_\infty$. Nekrashevych proved that such $V_d(G)$ are at least finitely presentable \cite[Theorem~5.9]{nekrashevych18}, and according to the last paragraph in the introduction of \cite{belk16} Bartholdi and Geoghegan have some unpublished results proving that they are indeed type $\F_\infty$. It seems unlikely our methods could yield this broad of a result, with the biggest immediate roadblock being the stabilizers. In a nutshell, the mere fact that $G$ has a finite nucleus (which might not be a subgroup) doesn't give us any natural type $\F_\infty$ subgroups $H$ to use to build stabilizers of a Stein--Farley-esque complex, and it is unclear where to hunt for another nice complex on which $V_d(G)$ acts with nice (finite? trivial?) stabilizers.

Since this is technically still open as of the writing of the present paper we will record it here as a question (though Bartholdi and Geoghegan's aforementioned unpublished work would indicate the answer is ``yes''):

\begin{question}
 Is $V_d(G)$ of type $\F_\infty$ for every finitely generated, contracting self-similar $G$?
\end{question}

It would also be very interesting to find an example of a finitely generated self-similar (non-contracting, non-$\F_\infty$) group $G\le \Aut(\tree_d)$ such that $V_d(G)$ is not of type $\F_\infty$. We have seen examples where $G$ is not even finitely presentable but $V_d(G)$ is of type $\F_\infty$ (like the R\"over group), so in general it seems that any negative finiteness properties would have to come from somewhere other than ``carrying over'' from $G$. This question of whether R\"over--Nekrashevych groups are always $\F_\infty$ is especially interesting since they are often simple or virtually simple (see \cite[Theorems~9.11 and~9.14]{nekrashevych04} for a precise statement), and to the best of our knowledge it is still an open problem to find simple groups with arbitrary finiteness properties in the world of Thompson-like groups (see also the question in the introduction of \cite{witzel16}). Note that the $\F_1$-but-not-$\F_2$ case is handled for example by the basilica Thompson group $T_B$ \cite{belk15,witzel16}, but even the $\F_2$-but-not-$\F_3$ case is open.

\begin{question}\label{quest:not_F_infty}
 Does there exist a finitely generated self-similar $G$ such that $V_d(G)$ is not of type $\F_\infty$?
\end{question}

We remark that, after this paper was written, Question~\ref{quest:not_F_infty} was solved affirmatively by the authors and Stefan Witzel \cite{skipper19}.

\bibliographystyle{alpha}
\newcommand{\etalchar}[1]{$^{#1}$}

\end{document}